\documentclass[reqno]{amsart}
\usepackage{amssymb}
\usepackage{mathrsfs}
\usepackage{pdfsync}

\usepackage{amsmath}
  \usepackage{paralist}
  \usepackage{graphics} 
  \usepackage{epsfig} 
 \usepackage[colorlinks=true]{hyperref}
\hypersetup{urlcolor=blue, citecolor=red}
\newtheorem{theorem}{Theorem}[section]
\newtheorem{corollary}{Corollary}

\newtheorem{lemma}[theorem]{Lemma}
\newtheorem{proposition}{Proposition}

\theoremstyle{definition}

\newtheorem{remark}{Remark}

\newcommand{\eps}[1]{{#1}_{\varepsilon}}

\def\EE{{\mathbb E}}
\def\FF{{\mathbb F}}

\def\R{{\mathbb R}}
\def\DD{{\mathscr D}}
\def\LL{{\mathscr L}}
\def\KK{{\mathscr K}}
\def\N{{\mathbb N}}
\def\C{{\mathbb C}}

\def\Z{{\mathbb Z}}

\def\E{{\mathcal E}}
\def\F{{\mathcal F}}

\def\K{{\mathcal K}}
\def\HH{{\mathcal H}}

\def\T{{\mathcal T}}
\def\L{{\mathcal L}}
\def\W{W}

\def\RR{{\mathcal R}}
\def\T{{\mathcal T}}
\def\eps{\varepsilon}
\def\3{{\ss}}

\def\HIC{$\HH^\infty$-calculus}
                                                 
\def\BUC{{\mathrm{BUC}}}

\def\sg{{\mathrm{sg}}}

\def\dd{{\mathrm d}}
\def\e{{\mathrm{e}}}

\def\heta{\hat{\eta}}
\def\hf{\hat{f}}
\def\hh{\hat{h}}
\def\hg{\hat{g}}
\def\dR{\dot{\R}}

\def\th{\tilde{h}}
\def\tg{\tilde{g}}

\def\teta{\tilde{\eta}}
\def\bareta{\bar{\eta}}

\def\del{\delta}

\def\rcvtp{\W^1_p(J,L^p(\dR^{n+1}))
        \cap L_p(J,\W^2_p(\dR^{n+1}))}
\def\rcs{W^{3/2-1/2p}_p(J,L_p(\R^n))
        \cap \W^1_p(J,W^{1-1/p}_p(\R^n))
        \cap L_p(J,W^{2-1/p}_p(\R^n))}

\def\rcftp{L_p(J,L_p(\dR^{n+1}))}
\def\rcg{W^{1-1/2p}_p(J,L_p(\R^n))
        \cap L_p(J,W^{2-1/p}_p(\R^n))}
\def\rch{W^{1/2-1/2p}_p(J,L_p(\R^n))
        \cap L_p(J,W^{1-1/p}_p(\R^n))}

\def\rcvitp{W^{2-2/p}_p(\dR^{n+1})}
\def\rcsi{W^{2-2/p}_p(\R^n)}
\def\rcsisu{W^{4-3/p}_p(\R^n)}
\def\rcss{W^{3/2-1/2p}_p(J,L_p(\R^n))
        \cap \W^{1-1/2p}_p(J,W^2_p(\R^n))
        \cap L_p(J,W^{4-1/p}_p(\R^n))}
\def\rcsu{W^{2-1/2p}_p(J,L_p(\R^n))
        \cap W^1_p(J,W^{2-1/p}_p(\R^n))}
\def\rcssu{W^{2-1/2p}_p(J,L_p(\R^n))
        \cap L_p(J,W^{4-1/p}_p(\R^n))}

\title[Singular limits for the two-phase Stefan]
      {Singular limits for the two-phase Stefan problem}

\author[Jan Pr{\tiny\"u}{\ss}, J{\tiny\"u}rgen Saal, and Gieri Simonett]{}

\subjclass{Primary: 35R35, 35B65, 80A22; Secondary: 35K20.}
 \keywords{Stefan problem, free boundary problem, phase transition,
 singular limits, maximal regularity.}

\email{jan.pruess@mathematik.uni-halle.de}
\email{saal@csi.tu-darmstadt.de}
\email{gieri.simonett@vanderbilt.edu}

\thanks{The second author is supported by the Center of Smart Interfaces
at TU Darmstadt}

\begin{document}
\maketitle

\centerline{\emph{Dedicated to Jerry Goldstein 
on the occasion of his 70th anniversary}}
\vspace{1cm}

\centerline{\scshape Jan Pr\"u{\ss}}
\medskip
{\footnotesize
 \centerline{Martin-Luther-Universit\"at Halle-Witten\-berg}
   \centerline{Institut f\"ur Mathematik}
   \centerline{D-06120 Halle, Germany}
} 

\medskip

\centerline{\scshape J\"urgen Saal}
\medskip
{\footnotesize
 \centerline{Technische Universit\"at Darmstadt}
   \centerline{Center of Smart Interfaces}
   \centerline{64287 Darmstadt, Germany}
}

\medskip

\centerline{\scshape Gieri Simonett}
\medskip
{\footnotesize
 \centerline{Vanderbilt University}
   \centerline{Department of Mathematics}
   \centerline{Nashville, TN~37240, USA}
}

\bigskip

 \centerline{(Communicated by the associate editor name)}

\begin{abstract}
We prove strong convergence to singular limits for a linearized
fully inhomogeneous Stefan problem subject to surface tension and kinetic
undercooling effects. Different combinations of $\sigma \to \sigma_0$
and $\delta\to\delta_0$, where $\sigma,\sigma_0\ge 0$ and 
$\delta,\delta_0\ge 0$ denote surface tension and kinetic undercooling 
coefficients respectively, altogether lead to five different types
of singular limits. Their strong convergence is based on uniform
maximal regularity estimates.
\end{abstract}


\section{Introduction}

The aim of this note is to consider the fully inhomogeneous 
system 
\begin{equation}
\label{LTS}
        \left\{
        \begin{array}{r@{\quad=\quad}ll}
        (\partial_t-c\Delta)v
         & f
        & \mbox{in}\ J\times\dR^{n+1},\\
        \gamma v^\pm -\sigma\Delta_x\rho+\del\partial_t\rho& g 
	& \mbox{on}\ J\times\R^n,\\
        \partial_t\rho+[\![c\partial_y(v-a\rho_E)]\!] & h
        & \mbox{on}\ J\times\R^n,\\
        v(0) & v_0 & \mbox{in}\ \dR^{n+1},\\
        \rho(0) & \rho_0 & \mbox{in}\ \R^n,\\
        \end{array}
        \right.
\end{equation}
which represents a linear model problem for the two-phase Stefan problem
subject to surface tension and kinetic undercooling effects.
Here 
\[
v(t,x,y)=
\left\{
\begin{array}{rl}
v^+(t,x,y),&y>0,\\
v^-(t,x,y),&y<0,
\end{array}
\right.
\qquad x\in\R^{n},\ y\in\R\setminus\{0\},\ t\in J,
\]
denotes the temperature in the two bulk phases
$\R^{n+1}_\pm=\{(x,y);\ x\in\R^n,\  \pm y>0\}$, 
and we have set $\dR^{n+1}=\R^{n+1}_+\cup\R^{n+1}_-$
and $J=(0,T)$.
The function $\rho$ appearing in the boundary 
conditions describes the free interface, 
which is assumed to be given
as the graph of $\rho$. We also admit the possibility of two
different (but constant) diffusion coefficients $c_\pm$ in the 
two bulk phases.
The parameters $\sigma$ and $\delta$ are related to surface tension and
kinetic undercooling. The function $\rho_E$ is an extension
of $\rho$ chosen suitably for our purposes. Here it is always
determined through 
\begin{equation}
\label{2-rho_E}
    \left\{
        \begin{array}{r@{\quad=\quad}ll}
        (\partial_t-c\Delta)\rho_E 
                &0
        & \mbox{in}\ J\times\dR^{n+1},\\
        \gamma \rho_E^{\pm} & \rho & \mbox{on}\ J\times\R^n,\\
        \rho_E(0)  & e^{-|y|(1-\Delta_x)^{\frac12}}\rho_0
        & \mbox{in}\ \dR^{n+1}.\\
        \end{array}
        \right.
\end{equation}

Using this notation, let $[\![c\partial_y (v-\rho_E)]\!]$ 
denote the jump of the normal derivatives 
across $\R^n$, that is,
$$[\![c\partial_y (v-\rho_E)]\!]:=c_+\gamma\partial_y (v^+-\rho_E^+)
-c_-\gamma\partial_y (v^- -\rho_E^-),$$
where $\gamma$ denotes the trace operator.
The coefficient $a$ is supposed to be a function of $\delta$ and
$\sigma$, that is, $a_\pm:[0,\infty)^2\to\R$, 
$[(\del,\sigma)\mapsto a_\pm(\del,\sigma)]$. It is further assumed to 
satisfy the conditions
\begin{equation}\label{cont_cond_a}
	a_\pm\in C([0,\infty)^2,\R),
	\qquad\quad
	a_\pm(0,0)>0.
\end{equation}
Recall from \cite{PSS07} that the introduction of the additional
term '$a\rho_E$' with $a_\pm>0$ in the situation of the 
classical Stefan problem is motivated by
the following two facts: for suitably chosen $a$ (depending on
the trace of the initial value and $\partial_y\rho_E$) it can be guaranteed
that a certain nonlinear term remains small for small times.
On the other hand, the additional term '$a\rho_E$'
is exactly the device that renders sufficient regularity
for the linearized problem. Note that, concerning regularity,
this additional term is not required if surface tension or 
kinetic undercooling is present. However, in order to obtain
convergence in best possible regularity classes for the limit 
$\sigma,\delta \to0$, we keep the term
'$a\rho_E$' in all appearing systems.
Since the data may (in general even must; see Remark~\ref{clarirem})
depend on $\sigma$ and $\del$  
as well, $a$ is a function of these two parameters. The natural 
and necessary convergence
assumption \eqref{conv_ass_inhom_1} then implies that we can assume
that $a_\pm\in C([0,\infty)^2,\R)$.
This continuity will be important in deriving maximal
regularity estimates for related boundary operators; see the proof of
Proposition~\ref{reg_op_l}.

The results of this paper on system (\ref{LTS}) 
represent an essential step in the
treatment of singular limits for the nonlinear Stefan
problem on general geometries. This will be the topic of
a forthcoming paper.

To formulate our main results, 
let $W^s_p(\R^n)$, $s\ge 0$, $p\in(1,\infty)$, denote the 
Sobolev-Slobodeckij spaces, cf. \cite{Tri83} 
(see also Section~\ref{sec_lin_op}). 
Depending on the presence of surface tension and/or kinetic
undercooling we obtain different regularity classes for
$\rho$, the function describing the evolution of the free
interface. To formulate this in a 
precise way we define for $J=(0,T)$ and $\del,\sigma\ge0$,
\begin{equation}\label{def_param_dep_space}
	\EE^2_T(\del,\sigma)
	:=\left\{\rho\in\EE^2_T(0,0):
	   \del\|\rho\|_{\EE^2_T(1,0)}
	   +\sigma\|\rho\|_{\EE^2_T(0,1)}<\infty\right\},	
\end{equation}
equipped with the norm
\begin{equation}\label{norm_param_dep_space}
	\|\cdot\|_{\EE^2_T(\del,\sigma)}
	:= \|\cdot\|_{\EE^2_T(0,0)}
	   +\del\|\cdot\|_{\EE^2_T(1,0)}
	   +\sigma\|\cdot\|_{\EE^2_T(0,1)},
\end{equation}
and where
\begin{eqnarray*}
	\EE^2_T(0,0)\hspace{-2mm}&:=&\hspace{-2mm}\rcs,\\
	\EE^2_T(1,0)\hspace{-2mm}&:=&\hspace{-2mm}\rcsu,\\
	\EE^2_T(0,1)\hspace{-2mm}&:=&\hspace{-2mm}\rcss,
\end{eqnarray*}
equipped with their canonical norms.
For the different values of
$\del$ and $\sigma$ (i.e., $\del=\sigma=0$, or $\del>0 $ and $\sigma=0$,
or $\del=0$ and $\sigma>0$, or $\del$ and $\sigma>0$) we obtain 
four different regularity classes for $\rho$. 
This leads to the following
five types of singular limits for problem (\ref{LTS}):
\begin{itemize}
\item[(1)]$(\del,\sigma)\to (0,0)$, \ \ \  $\del,\sigma>0$,
\smallskip
\item[(2)]$(\del,\sigma)\to (\del_0,0)$, \ \ for $\del_0>0$ fixed,
\smallskip
\item[(3)]$(\del,\sigma)\to (0,\sigma_0)$, \ \ for $\sigma_0>0$ fixed, 
\smallskip
\item[(4)]$(\del,0)\to (0,0)$,
\smallskip
\item[(5)]$(0,\sigma)\to (0,0)$.
\end{itemize}
Our main result, Theorem~\ref{conv_main_lin}, 
covers convergence results for all these limits.

In the sequel 
\[
	\sg(t):=\left\{\begin{array}{rl}
	               1,&t>0,\\
		       0,& t=0,\\
	               -1,&t<0,
		       \end{array}
		       \right.
\]
will denote the sign function. Our first main result is on 
maximal regularity.
Here we refer to Section 2 for the definition of the space of data $\FF_T(\del,\sigma)$.
The essential difference to corresponding results in previous publications is the uniformness 
of the estimates with respect to the parameters $\del$ and $\sigma$. 
\bigskip
\begin{theorem}
\label{main_lin}
Let $3<p<\infty$, $R,T>0$, $0\le\del,\sigma\le R$,
and suppose that $a=a(\del,\sigma)$ is a function 
satisfying the conditions in \eqref{cont_cond_a}.
There exists a unique solution
\[
	(v,\rho,\rho_E)=(v^{(\del,\sigma)},\rho^{(\del,\sigma)},\rho^{(\del,\sigma)}_E)\in \EE_T(\del,\sigma)
\]
for \eqref{LTS}--\eqref{2-rho_E} if and only if
the data satisfy
\begin{equation}\label{assum_ml_a}
        (f,g,h,v_0,\rho_0)\in \FF_T(\del,\sigma),
\end{equation}
\begin{equation}\label{assum_ml_b}
        \gamma v_0^{\pm}-\sigma\Delta_x\rho_0
	      +\del\left(h(0)-[\![c\gamma\partial_y
	      (v_0-a\e^{-|y|(1-\Delta_x)^{1/2}}\rho_0)]\!]\right)
	      =g(0),
\end{equation}
and, if $\del=0$, also that 
\begin{equation}\label{assum_ml_c}
	      \sigma(h(0)-[\![c\gamma\partial_y v_0]\!])
	      \in W^{2-6/p}_p(\R^n).
\end{equation}
Furthermore, the solution satisfies the estimate
\begin{eqnarray}
	\|(v,\rho,\rho_E)\|_{\EE_T(\del,\sigma)}
	&\le& C\left(
	\|(f,g,h,v_0,\rho_0)\|_{\FF_T(0,0)}
	+(\del+\sigma)\|\rho_0\|_{W^{4-3/p}_p(\R^n)}
	\right.\nonumber\\
	&&\left.\strut+
	\sigma\|h(0)-[\![c\gamma\partial_y v_0]\!]\|_{W^{2-6/p}_p(\R^n)}
	\right),
	\label{max_reg_ineq_inhom}
\end{eqnarray}
where the constant $C>0$ is independent of $(\del, \sigma)\in [0,R]^2$. 
\end{theorem} 
\noindent
Our main result on convergence of singular limits is
\begin{theorem}\label{conv_main_lin}
Let $3<p<\infty$, $R,T>0$, $0\le\del_0\le\del\le R$,
$0\le\sigma_0\le\sigma\le R$, and  $a=a(\del,\sigma)$ be a function 
satisfying the conditions in \eqref{cont_cond_a}. 
Set $\mu:=(\del,\sigma)$,
$\mu_0:=(\del_0,\sigma_0)$, and 
$I_0:=[\del_0,R]\times[\sigma_0,R]$. Suppose that
\[
	((f^\mu,g^\mu,h^\mu,v_0^\mu,\rho_0^\mu))_{\mu\in I_0}
	\subset\FF_T(\mu)
\]
and that the compatibility conditions \eqref{assum_ml_b}
and \eqref{assum_ml_c} in 
Theorem~\ref{main_lin} are satisfied for each 
$\mu\in I_0$. 
Furthermore, denote 
by $(v^\mu,\rho^\mu,\rho_E^\mu)$ the 
solution of \eqref{LTS}--\eqref{2-rho_E} 
given in Theorem~\ref{main_lin} that corresponds to the
parameter $\mu=(\del,\sigma)\in I_0$.
Under the convergence assumptions that 
\begin{equation}
\label{conv_ass_inhom_1}
	(f^\mu,g^\mu,h^\mu,v_0^\mu,\rho_0^\mu)
	\to (f^{\mu_0},g^{\mu_0},h^{\mu_0},v_0^{\mu_0},
	\rho_0^{\mu_0})\quad\mbox{in}\quad\FF_T(\mu_0),
\end{equation}
and, if $\del_0=0$, that
\begin{equation}
\label{conv_ass_inhom_2}
	\sigma\left(h^\mu(0)-[\![c\gamma\partial_y v_0^\mu ]\!]\right)
	\to \sigma_0\left(
	h(0)^{\mu_0}-[\![c\gamma\partial_y v_0^{\mu_0}]\!]\right)
	\quad\mbox{in}\quad W^{2-6/p}_p(\R^n)
\end{equation}
and, if $\del_0=\sigma_0=0$ and $\del>0$,
also that
\begin{equation}
\label{conv_ass_inhom_3}
	(\del+\sigma) \rho_0^\mu
	\to 0
	\quad\mbox{in}\quad W^{4-3/p}_p(\R^n)
\end{equation}
on the data,
we obtain strong convergence of the solution, i.e.,
we have that
\begin{equation}\label{conv_inhom_lin}
	(v^\mu,\rho^\mu,\rho_E^\mu)\to
	(v^{\mu_0},\rho^{\mu_0},\rho_E^{\mu_0})
	\quad\mbox{in}\quad\EE_T(\mu_0).
\end{equation}
\end{theorem}
\bigskip
\begin{remark}\label{clarirem}
(a) \ Note that for $\del>0$ condition 
\eqref{assum_ml_c} follows 
automatically from condition \eqref{assum_ml_b}.\\[1mm]
(b) \ Conditions \eqref{conv_ass_inhom_1} and \eqref{conv_ass_inhom_3}
for the last component
are obviously satisfied for a fixed initial interface
in $\rcsisu$, i.e., if we assume 
$\rho_0^\mu=\rho_0\in \rcsisu$ for all $\mu\in I_0$.
But observe that, due to condition \eqref{assum_ml_b},
it is not possible to fix $v_0$ as well.\\[1mm]
(c) \ In analogy to (a) note that for $\del_0>0$ assumption 
\eqref{conv_ass_inhom_2} follows automatically from 
\eqref{assum_ml_b} and \eqref{conv_ass_inhom_1}.
Also observe that in the case 
$\delta=\delta_0=\sigma_0=0$
condition \eqref{conv_ass_inhom_3} follows automatically
from conditions \eqref{assum_ml_b} and \eqref{conv_ass_inhom_1}.
\\[1mm] 
(d) \ In the case $\del_0=\sigma_0=0$ conditions 
\eqref{conv_ass_inhom_2} 
and \eqref{conv_ass_inhom_3}
express that $\|\rho_0^\mu \|$ and
$\|h(0)^{\mu}-[\![\gamma\partial_y v_0^{\mu}]\!]\|$ might blow up in 
$\rcsisu$ and $W^{2-6/p}_p(\R^n)$ respectively, but slower than 
$\sigma$ and $\del$ tend to zero.
This seems to be natural in view of the fact that we do not have
$\rho_0^{(0,0)}=\rho^{(0,0)}|_{t=0}\in \rcsisu$
and 
\[
h(0)^{(0,0)}-[\![c\gamma\partial_y v_0^{(0,0)}]\!]
=\partial_t\rho^{(0,0)}|_{t=0}\in W^{2-6/p}_p(\R^n)
\]
from the regularity of sol utions in the situation of the
classical Stefan problem. 
\end{remark}
The Stefan problem is a model for phase transitions in liquid-solid systems
that has attracted considerable attention over the last decades.
We refer to the recent publications 
\cite{EPS03,PSS07,PS08,PSZ11,PSW11} by the authors,
and the references contained therein,
for more background information on the Stefan problem.

Previous results concerning singular limits for the Stefan problem with surface tension and kinetic undercooling
are contained in \cite{BaDe92,Yo96}.
Our work extends these results in several directions:
we obtain sharp regularity results (for the linear model problems),
we can handle all the possible combinations of singular limits, and we obtain convergence in the best
possible regularity classes.

Our approach relies on the powerful theory of maximal $L_p$-regularity, 
${\mathcal H}^\infty$-func\-tion\-al calculus, and ${\mathcal R}$-boundedness,
see for instance \cite{DHP03,KW04} for a systematic introduction.


\section{Maximal regularity}
\label{sec_lin_op}
\noindent
First let us introduce suitable function spaces.
Let $\Omega\subseteq\R^m$ be open and $X$ be 
an arbitrary Banach space.  
By $L_p(\Omega;X)$ and $H^s_p(\Omega;X)$, 
for $1\le p\le\infty$ and $s\in\R$, we denote the $X$-valued Lebegue and
the Bessel potential space of order $s$, respectively.
We will also frequently make use of the fractional 
Sobolev-Slobodeckij
spaces $W^s_p(\Omega;X)$, $1\le p< \infty$,
$s\in\R\setminus\Z$, with norm
\begin{equation}
\label{Slobodeskii}
        \|g\|_{W^s_p(\Omega;X)}
        =\|g\|_{W^{[s]}_p(\Omega;X)}
        +\sum_{|\alpha|=[s]}
         \left(\int_\Omega\int_\Omega
         \frac{\|\partial^\alpha g(x)-\partial^\alpha g(y)\|^p_{X}}
         {|x-y|^{n+(s-[s])p}}\dd x\dd y\right)^{\!1/p}{\hskip-4mm},
\end{equation}
where $[s]$ denotes the largest integer smaller than $s$.
Let $T\in(0,\infty]$ and $J=(0,T)$.
We set
\[
        _0W^s_p(J,X):=\left\{
        \begin{array}{l}
        \{u\in W^s_p(J,X): u(0)=u'(0)=\ldots=u^{(k)}(0)=0\},\\[3mm]
        \mbox{if}\quad 
        k+\frac1{p}<s<k+1+\frac1{p}, \ k\in\N\cup\{0\},\\[3mm]
        W^s_p(J,X),\quad \mbox{if}\quad s<\frac1{p}.
        \end{array}
        \right.
\]
The spaces $_0H^s_p(J,X)$ are defined analogously.
Here we remind that $H^k_p=W^k_p$ for $k\in\Z$ and $1<p<\infty$,
and that $W^s_p=B^s_{pp}$ for $s\in\R\setminus\Z$.
We refer to \cite{Tri78,Tri83} for more information.

Before turning to the proofs of our main results,
we add the following remarks on the linear two-phase Stefan problem 
\eqref{LTS} and the particularly chosen 
extension $\rho_E$ determined by equation \eqref{2-rho_E}.
\begin{remark}
\label{Bemerkungen}
(a) \ \eqref{LTS}--\eqref{2-rho_E}
constitutes a coupled system of equations,
with the functions $(v,\rho,\rho_E)$
to be determined.
We will in the sequel often just refer to a solution
$(v,\rho)$ of \eqref{LTS} with the understanding
that the function $\rho_E$ also has to be determined.
\goodbreak
\smallskip
\noindent
(b) \ 
Suppose $\rho\in W^{1-1/2p}_p(J,L_p(\R^n))\cap 
L_p(J, W_p^{2-1/p}(\R^n))$ and 
$\rho_0\in W^{2-3/p}_p(\R^n)$ is given
such that $\rho(0)=\rho_0$.
Then 
the diffusion equation \eqref{2-rho_E}
admits a unique solution
$$\rho_E\in \W^1_p(J,L^p(\dR^{n+1}))\cap L_p(J,\W^2_p(\dR^{n+1})).$$
This follows, for instance, from \cite[Proposition 5.1]{EPS03},
thanks to
\[
	e^{-|y|(1-\Delta_x)^{\frac12}}\rho_0
	\in W^{2-2/p}_p(\dR^{n+1}).
\]
(c) \ The solution $\rho_E(t,\cdot)$ of equation \eqref{2-rho_E}
provides an extension of
$\rho(t,\cdot)$ to $\dR^{n+1}$.
We should remark that there are many possibilities to
define such an extension.
The chosen one is the most convenient for our purposes.
We also remark that we have great freedom for the extension 
of $\rho_0$.
\end{remark}
Let $T\in(0,\infty]$
and set $J=(0,T)$. By $\FF_T$ we always mean the space
of given data $(f,g,h,v_0,\rho_0)$, i.e., $\FF_T$ is
given by
\[
	\FF_T=\FF_T^1\times\FF_T^2\times\FF_T^3
	      \times\FF^4_T\times\FF^5_T(\del,\sigma),
\]
where
\begin{eqnarray*}
	\FF^1_T&=&\rcftp,\\
	\FF^2_T&=&\rcg,\\
	\FF^3_T&=&\rch\\
	\FF^4_T&=&\rcvitp\\
	\FF^5_T(\del,\sigma)&
	=&W^{2-2/p+\sg(\del+\sigma)(2-1/p)}_p(\R^n).
\end{eqnarray*}
Analogously, we denote by $\EE_T$ the space of the solution
$(v,\rho,\rho_E)$. As was already pointed out in the introduction, 
we have, depending on the
values of $\del$ and $\sigma$, four different types of spaces. For this
reason we set
\[
	\EE_T(\del,\sigma)=\EE^1_T\times\EE^2_T(\del,\sigma)
	\times\EE^1_T
	\quad (\del,\sigma\ge 0),
\]
with
\[
	\EE^1_T=\rcvtp,
\]
and with $\EE^2_T(\del,\sigma)$ as defined in 
\eqref{def_param_dep_space} and equipped with the 
parameter dependent norm given in 
\eqref{norm_param_dep_space}. 
Note that then the norm in $\EE_T(\del,\sigma)$ is given by
\[
	\|(v,\rho,\rho_E)\|_{\EE_T(\del,\sigma)}
	= \|(v,\rho,\rho_E)\|_{\EE_T(0,0)}
	   +\del\|\rho\|_{\EE^2_T(1,0)}
	   +\sigma\|\rho\|_{\EE^2_T(0,1)}
\]
for $(v,\rho,\rho_E)\in\EE_T(\del,\sigma)$.
For fixed $\del,\sigma>0$ by interpolation it can be shown that
\[
	\EE^2_T(\del,\sigma)
	=\rcssu
\]
in the sense of isomorphisms.
We remark that $\EE^2_T(\del,\sigma)$ is the correct regularity class for the
free surface if both, surface tension and kinetic undercooling
are present. The space $\EE^2_T(0,\sigma)$ or $\EE^2_T(\del,0)$
is the proper class if just surface tension or just kinetic
undercooling, respesctively, is present. Finally, 
$\EE^2_T(0,0)$ is the correct class if both of them 
are missing, i.e., $\EE^2_T(0,0)$ is the regularity class 
in the situation of the classical Stefan problem.

The corresponding spaces with zero time trace at the origin 
are denoted by ${_0\FF^1_T}$, ${_0\EE^1_T}$, 
${_0\EE^2_T(\del,\sigma)}$, and so on, that is,
\begin{eqnarray*}
	{_0\FF^2_T}&=&{_0\rcg}\quad\mbox{or}\\
	{_0\EE^1_T}&=&{_0\rcvtp},
\end{eqnarray*}
for instance. Moreover, we set 
\begin{eqnarray*}
	{_0\FF_T}&:=&{\FF^1_T}\times{_0\FF^2_T}\times{_0\FF^3_T},\\
	{_0\EE_T(\del,\sigma)}
	&:=&{_0\EE^1_T}\times{_0\EE^2_T(\del,\sigma)}
	\times{_0\EE^1_T}.\\
\end{eqnarray*}
\subsection{Zero time traces}
\label{sec_ztt}
We will first consider the special case
that 
\[
	(h(0),g(0),v_0,\rho_0)=(0,0,0,0).
\]
This allows us to derive an explicit representation
for the solution of \eqref{LTS}--\eqref{2-rho_E}.
\begin{theorem}
\label{homog_ini}
Let $p\in (3,\infty)$, $T,R>0$, $0\le\del,\sigma\le R$, 
and set $J=(0,T)$. 
Suppose that 
\[
	(f,g,h)\in {_0\FF_T}
\]
and that the function $a=a(\del,\sigma)$ satisfies
the conditions in \eqref{cont_cond_a}.
Then there is a unique solution 
\[
	(v,\rho, \rho_E)=(v^\mu,\rho^\mu, \rho_E^\mu)\in{_0\EE_T(\del,\sigma)} 
\]
of \eqref{LTS}--\eqref{2-rho_E}
satisfying
\begin{equation}\label{max_reg_ineq}
	\|(v,\rho,\rho_E)\|_{_0\EE_T(\del,\sigma)}
	\le C\|(f,g,h)\|_{_0\FF_T}
\end{equation}
with $C>0$ independent of the data, the parameters $(\del,\sigma)\in[0,R]^2$, 
and $T\in(0,T_0]$ for fixed $T_0>0$. 
\end{theorem}
\begin{proof}
(i) In order to be able to apply the Laplace transform
in $t$, we consider the modified set
of equations
\begin{equation}
\label{LOS-infty}
        \left\{
        \begin{array}{r@{\quad=\quad}ll}
        (\partial_t+\kappa-c\Delta) u
         & f
        & \mbox{in}\ (0,\infty)\times\dR^{n+1},\\
        \gamma u^{\pm} -\sigma\Delta_x\eta
	+\del(\partial_t+\kappa)\eta& g 
	& \mbox{on}\ (0,\infty)\times\R^n,\\
        (\partial_t+\kappa)\eta+[\![c\gamma\partial_y(u-a\eta_E)]\!] & h
        & \mbox{on}\ (0,\infty)\times\R^n,\\
        u(0) & 0 & \mbox{in}\ \dR^{n+1},\\
        \eta(0) & 0 & \mbox{in}\ \R^n,\\
        \end{array}
        \right.
\end{equation}
and
\begin{equation}
\label{rho_E-infty}
    \left\{
        \begin{array}{r@{\quad=\quad}ll}
        (\partial_t+\kappa-c\Delta)\eta_E 
                &0
        & \mbox{in}\ (0,\infty)\times\dR^{n+1}\,\\
        \gamma \eta_E^{\pm} & \eta 
        & \mbox{on}\ (0,\infty)\times\R^n,\\
        \eta_E(0)  & 0
        & \mbox{in}\ \dR^{n+1},\\
        \end{array}
        \right.
\end{equation}
for the unknown functions $(u,\eta,\eta_E)$ and 
for a fixed number $\kappa\ge 1$
to be chosen later.
We claim that system \eqref{LOS-infty}--\eqref{rho_E-infty}
admits for each $(f,g,h)\in{_0\FF_\infty}$ a unique solution 
\[
	(u,\eta,\eta_E)\in{_0\EE_\infty(\del,\sigma)}
\]
satisfying inequality (\ref{max_reg_ineq}) in the corresponding
norms for $T=\infty$.
\goodbreak
\smallskip
\noindent
(ii)
In the following, the symbol 
$\ \hat{}\ $ denotes the Laplace transform w.r.t.\
$t$ combined with the Fourier transform w.r.t.\ the tangential space
variable $x$.
Applying the two transforms to 
equation \eqref{rho_E-infty} yields 
\begin{equation}
\label{t-rho_E}
        \left\{
        \begin{array}{r@{\quad=\quad}l}
        (\omega^2 - c\partial_y^2) \widehat{\eta_E}(y)
         & 0,
        \quad y\in\dR, \\
        \widehat{\eta_E}^{\pm}(0) & \hat\eta,\\
            \end{array}
        \right.
\end{equation} 
where we set 
\begin{eqnarray*}
        \omega
        &=&\omega(\lambda,|\xi|,y)
	   =\sqrt{\lambda+\kappa+c(y)|\xi|^2},\\ 
           \omega_{\pm}
        &=&\omega_{\pm}(\lambda,|\xi|)
           =\sqrt{\lambda+\kappa+c_{\pm}|\xi|^2}. 
\end{eqnarray*}
with $c(y)=c_\pm$ for $(\pm y)>0$.
Equation \eqref{t-rho_E} can readily be solved to the result
\begin{equation}
\label{hat-rho_E}
\widehat{\eta_E}(y)=e^{-\frac{\omega}{\sqrt{c}}|y|}\hat\eta.
\end{equation}
Next, applying the transforms to \eqref{LOS-infty}
we obtain
\begin{equation}
\label{tsp}
        \left\{
        \begin{array}{r@{\quad=\quad}l}
        (\omega^2 -c\partial_y^2) \hat u(y)
         & \hf(y),
        \quad y\in\dR, \\
        \hat{u}^{\pm}(0) +\sigma|\xi|^2\heta+\del(\lambda+\kappa)
	\heta& \hg,\\
        (\lambda+\kappa)\heta
        +[\![c\partial_y(\hat{u}-a\widehat{\eta_E})(0) ]\!] & \hh.
        \end{array}
        \right.
\end{equation}
By employing the 
fundamental solution  
\[
        k_{\pm}(y,s):=\frac1{2\omega_{\pm}\sqrt{c_{\pm}}}
                (\e^{-\omega_{\pm}\left|y-s\right|/
                \sqrt{c_{\pm}}}
                -\e^{-\omega_{\pm}\left(y+s\right)/
                \sqrt{c_{\pm}}}),\quad y,s>0
\]
of the operator $(\omega_{\pm}^2-c_{\pm}\partial_y^2)$,
we make for $\hat u^{\pm}$ the ansatz
\begin{equation}
\label{u^+,u^-}
\begin{split}
        \hat u^+(y)&\ =\ \int_0^\infty k_+(y,s)\hf^+(s)\dd s
	               -\e^{\omega_+ y/\sqrt{c_+}}
		       (\sigma|\xi|^2\heta+\del(\lambda+\kappa)
		       \heta-\hg),
                      \quad y>0, \\
        \hat u^-(y)&\ =\ \int_0^\infty k_-(-y,s)\hf^-(-s)\dd s
	               -\e^{\omega_-y/\sqrt{c_-}}
		       (\sigma|\xi|^2\heta+\del(\lambda+\kappa)
		       \heta-\hg),
                      \quad y<0.
\end{split}
\end{equation}
A simple computation shows that
\begin{equation*}
\begin{split}
        \partial_y\hat u^+(0)
        &\ =\ \frac1{c_+}
           \int_0^\infty\e^{-\omega_+ s/\sqrt{c_+}}
           \hf^+(s)\dd s
	   +\frac{\omega_+}{\sqrt{c_+}}(\sigma|\xi|^2\heta
	   +\del(\lambda+\kappa)\heta-\hg)
           \qquad\mbox{and}\\
        \partial_y\hat u^-(0)
        &\ =\ -\frac1{c_-}
           \int_0^\infty\e^{-\omega_- s/\sqrt{c_-}}
           \hf^-(-s)\dd s
	   -\frac{\omega_-}{\sqrt{c_-}}(\sigma|\xi|^2\heta
	   +\del(\lambda+\kappa)\heta-\hg).
\end{split}
\end{equation*}
Inserting this and
the fact that $\partial_y \widehat{\eta_E}^{\pm}(0)
=\mp\frac{\omega_{\pm}}{\sqrt{c_{\pm}}}\hat\eta$
in the third line of (\ref{tsp}) yields
\begin{equation}
\label{form_rho}
\begin{split}
        \heta=&\frac1{m}
             \biggl(\hh-\int_0^\infty \e^{-\omega_+ s/\sqrt{c_+}}
             \hf^+(s)\dd s
             -\int_0^\infty \e^{-\omega_- s/\sqrt{c_-}}
             \hf^-(-s)\dd s\biggr.\\
	     &\biggl.\strut + \sqrt{c_+}\omega_+\hg
	     +\sqrt{c_-}\omega_-\hg
	     \biggr),
\end{split}
\end{equation}
with
\begin{eqnarray}
\label{m}
        m(\lambda,|\xi|)&=&\lambda+\kappa
	     +(\sigma|\xi|^2+\del(\lambda+\kappa))
	     \left(\sqrt{c_+}\omega_+(\lambda,|\xi|)
             +\sqrt{c_-}\omega_-(\lambda,|\xi|)\right)\nonumber\\
	     &&\strut + a_+\sqrt{c_+}\omega_+(\lambda,|\xi|)
             +a_-\sqrt{c_-}\omega_-(\lambda,|\xi|)\,.
\end{eqnarray}
\smallskip\\
(iii)
In order to show the claimed regularity for the Laplace Fourier 
inverse of the 
representation $(\hat u,\hat\eta)$ we first show regularity
properties of the symbols involved.
To this end let us introduce the operators that 
correspond to the time derivative and the Laplacian in tangential
direction.   
Let $r,s\ge 0$ and 
\[ 
        \F,\K\in\{H,W\}. 
\]
Then by $\K^s_p$ we either mean the space $H^s_p$
or the space $W^s_p$.
On the space
${_0\F}^r_p(\R_+,\K^s_p(\R^n))$ we define
\begin{equation}\label{dom_g}
        Gu=\partial_t u,
           \quad u\in \DD(G)={_0\F}^{r+1}_p(\R_+,\K^s_p(\R^n)),
\end{equation}
and 
\[
        D_n u=-\Delta u
              \quad u\in \DD(D_n)
              ={_0\F}^r_p(\R_+,\K^{s+2}_p(\R^n)),
\] 
that is, $D_n$ denotes the canonical
extension to ${_0\F}^r_p(\R_+,\K^s_p(\R^n))$ 
of $-\Delta$ in $\K^s_p(\R^n)$.
Note that
\begin{equation}\label{rh_infty_g}
	G\in\RR\HH^\infty({_0\F}^r_p(\R_+,\K^s_p(\R^n)))
	\quad\mbox{with}\quad
	\phi^{R,\infty}_{G}=\pi/2
\end{equation}
and
\begin{equation}\label{rh_infty_dn}
	D_n\in\RR\HH^\infty({_0\F}^r_p(\R_+,\K^s_p(\R^n)))
	\quad\mbox{with}\quad
	\phi^{R,\infty}_{D_n}=0,
\end{equation}
i.e.\ both, $G$ and $D_n$ admit an $\RR$-bounded \HIC{}
with $\RR\HH^\infty$-angle $\phi^{R,\infty}_{G}=\pi/2$ and 
$\phi^{R,\infty}_{D_n}=0$, respectively.
Recall that an operator $A$ admits an $\RR$-bounded \HIC{}
with $\RR\HH^\infty$-angle $\phi^{R,\infty}_A$,
if it admits a bounded \HIC{} and if 
\[
        \RR\big(\big\{h(A):h\in H^\infty(\Sigma_\phi),
        \|h\|_\infty\le 1\big\}\big)
        <\infty
\]
for each $\phi>\phi^{R,\infty}_A$,
where $\RR(\T)$ denotes the $\RR$-bound of an operator family
$\T\subset \L(X)$ for a Banach space $X$, see \cite{DHP03,KW04} for additional information.

The inverse transform of the occuring symbols
can formally be regarded as functions of $G$ and $D_n$.
We first consider the symbol $\omega_\pm$. 
The corresponding operator is formally given by
\begin{equation}\label{def_f}
        F_{\pm}=(G+\kappa+c_{\pm}D_n)^{1/2}.
\end{equation}
\begin{lemma}\label{dom_f}
Let $1<p<\infty$ and $r,s\ge 0$. Then we have that
\[
	F_{\pm}:\DD(F_{\pm})\to{_0\F}^r_p(\R_+,\K^s_p(\R^n))
\]	
with
\[
	\DD(F_{\pm})={_0\F}^{r+1/2}_p(\R_+,\K^s_p(\R^n))
             \cap{_0\F}^r_p(\R_+,\K^{s+1}_p(\R^n)),
\]
is closed and invertible,
where we set ${\F}=H$ in case $2r\in\N$.
\end{lemma}
\begin{proof}
The assertion follows from \cite[Proposition 2.9 and Lemma 3.1]{MS11}.
\end{proof}
Next we show closedness and invertibility of the operator 
\begin{equation}
\label{L}
        L:=G+\kappa+(\sigma D_n+\del(G +\kappa))
	\left(\sqrt{c_+}F_+ +\sqrt{c_-}F_-
	   \right)+a_+\sqrt{c_+}F_+ +a_-\sqrt{c_-}F_-\,,
\end{equation}  
associated with the symbol $m$ introduced in \eqref{m}, in the space 
${_0\F}^r_p(\R_+,\K^s_p(\R^n))$.
We will prove invertibility of $L$ and derive uniform estimates 
with respect to the parameters $(\delta,\sigma)$ in various adapted norms.
In view of (\ref{rh_infty_g}), (\ref{rh_infty_dn}),
and by the Theorem of Kalton and
Weis \cite[Theorem~4.4]{KW01} it essentially remains
to show the holomorphy and the boundedness of the symbols
regarded as functions of $\lambda$ and $|\xi|^2$ on certain
complex sectors.

In order to obtain these estimates, 
the following simple lemma will be useful.
\goodbreak
\begin{lemma}\label{triangle}
Let $G\subseteq\C^n$ be a domain. Let $f_1,f_2:G\to\C$ be 
functions such that $f_1(z)\neq 0$ for $z\in G$. Then
the following statements are equivalent:
\begin{itemize}
\item[(i)] $\displaystyle -1\ \not\in \ \overline{\frac{f_2}{f_1}(G)}$.
\smallskip
\item[(ii)] There exists a $c_0>0$ such that
\[
	|f_1(z)+f_2(z)|\ge c_0(|f_1(z)|+|f_2(z)|),
	\quad z\in G.
\]
\end{itemize}
\end{lemma}
\begin{proof}
We set
\[
	g:G\to\R,\quad
	g(z):=\frac{|f_1(z)+f_2(z)|}{|f_1(z)|+|f_2(z)|},
	\quad z\in G,
\]
which is a well defined function.
Observe that (ii) is equivalent to saying that 
$0\not\in \overline{g(G)}$.
By contradiction arguments it is not difficult to show
that this relation is equivalent to condition (i).
\end{proof}
\medskip
\begin{remark}
The assumption $f_1(z)\neq 0$ for $z\in G$ 
is just for technical reasons and can be removed.
\end{remark}
\noindent
Now we prove closedness and invertibility of $L$.
\begin{proposition}
\label{reg_op_l}
Let $1<p<\infty$, $r,s\ge 0$, $R>0$, $(\del,\sigma)\in [0,R]^2$, 
and $\F,\K\in\{H,W\}$. Suppose that $a$ is a function 
satisfying condition \eqref{cont_cond_a}. Then there is a number
$\kappa\ge 1$ such that
\begin{eqnarray*}
        \DD(L)&=&{_0\F}^{r+1+\sg(\del)/2}_p(\R_+,\K^s_p(\R^n))
                 \cap{_0\F}^{r+1}_p(\R_+,\K^{s+\sg(\del)}_p(\R^n))\\
              && \strut
	         \cap{_0\F}^{r+1/2}_p(\R_+,\K^{s+2\sg(\sigma)}_p(\R^n))
                 \cap{_0\F}^r_p(\R_+,\K^{s+1+2\sg(\sigma)}_p(\R^n))
\end{eqnarray*}
and $L:\DD(L)\to {_0\F}^r_p(\R_+,\K^s_p(\R^n))$ is invertible.
Furthermore, 
\begin{eqnarray*}
	&&\sigma\|D_n (G+1)^{1/2}L^{-1}\|_0
	+\sigma\|D_n^{3/2}L^{-1}\|_0\\
	&&\strut+\del\|(G+1)^{3/2}L^{-1}\|_0
	+\del\|D_n^{1/2}(G+1)L^{-1}\|_0
	\strut+\|L^{-1}\|_1\le C
\end{eqnarray*}
with $C>0$ independent of $(\del,\sigma)\in[0,R]^2$,
where $\|\cdot\|_0$ denotes the norm in
\[
	\LL\left({_0\F}^r_p(\R_+,\K^s_p(\R^n))\right),
\]
and $\|\cdot\|_1$ the norm in
\[
	\LL\left({_0\F}^r_p(\R_+,\K^s_p(\R^n)),
	{_0\F}^{r+1}_p(\R_+,\K^s_p(\R^n))
        \cap{_0\F}^r_p(\R_+,\K^{s+1}_p(\R^n))\right).
\]
\end{proposition}
\begin{proof}
Let $\varphi_0\in(0,\pi/2)$ and $\varphi\in(0,\varphi_0)$.
By a compactness and homogeneity argument it easily follows that
\begin{eqnarray}
	|\omega_{\pm}(\lambda,z)|
	&=&
	|\sqrt{\lambda+\kappa+c_\pm z}|\nonumber\\
	&\ge& c_0\left(
	\sqrt{|\lambda|}+\sqrt{\kappa}+c_\pm \sqrt{|z|}\right)
	\label{est_omega}
\end{eqnarray}
for all $(\lambda,z,\kappa)\in\Sigma_{\pi-\varphi_0}
\times\Sigma_{\varphi}\times [1,\infty)$ and some $c_0>0$.

In the following we let  $\varphi_0\in(\pi/3,\pi/2)$ and 
$\varphi\in(0,\varphi_0-\pi/3)$.
Note that by condition \eqref{cont_cond_a} on $a$ there
exist $\del^*,\sigma^*>0$ and $M,c_0>0$ such that
\begin{equation}\label{a_greater_zero}
	a_\pm(\del,\sigma)\ge c_0 \qquad\quad(
	(\del,\sigma)\in [0,\del^*]\times[0,\sigma^*])
\end{equation}
and
\begin{equation}\label{a_smaller_m}
	|a_\pm(\del,\sigma)|\le M \qquad\quad(
	(\del,\sigma)\in [0,R]\times[0,R]).
\end{equation}

First assume that \eqref{a_greater_zero} is satisfied, i.e.,
that $(\del,\sigma)\in [0,\del^*]\times[0,\sigma^*]$.
Let $m$ be as given in (\ref{m}). We consider the function 
\begin{equation*}
\begin{split}	
        &f:\Sigma_{\pi-\varphi_0}\times\Sigma_{\varphi}
	\times [0,\del^*]\times[0,\sigma^*]\times[1,\infty) \to \C,\\
	&(\lambda,z,\del,\sigma,\kappa)\mapsto 
	f(\lambda,z,\del,\sigma,\kappa):=m(\lambda,z)
	:=f_1(\lambda,z,\del,\sigma,\kappa)
	+f_2(\lambda,z,\del,\sigma,\kappa),
\end{split}
\end{equation*}
with
\begin{eqnarray*}
        f_1(\lambda,z,\sigma,\del,\kappa)
	&:=&(\lambda+\kappa)\left[\del(\sqrt{c_+}\omega_+(\lambda,z) 
	+\sqrt{c_-}\omega_-(\lambda,z))+1\right],\\
        f_2(\lambda,z,\sigma,\del,\kappa)
	&:=& m(\lambda,z)-f_1(\lambda,z,\sigma,\del,\kappa)\\
	&=& \sigma z
	     \left(\sqrt{c_+}\omega_+(\lambda,z)
             +\sqrt{c_-}\omega_-(\lambda,z)\right)\\
	&&  \strut + a_+(\del,\sigma)\sqrt{c_+}\omega_+(\lambda,z)
             +a_-(\del,\sigma)\sqrt{c_-}\omega_-(\lambda,z).
\end{eqnarray*}
Note that by our choice of the angle $\varphi$ for 
$(\lambda,z,\del,\sigma,\kappa)
\in\Sigma_{\pi-\varphi_0}\times\Sigma_{\varphi}
\times[0,\del^*]\times[0,\sigma^*]\times[1,\infty)$ 
with $\arg \lambda\ge 0$ there exists an $\eps>0$
such that
\[
        \pi-\varphi_0\ge \frac{\pi-\varphi_0}{2}+\varphi
	\ge \arg \sigma z\sqrt{\lambda + \kappa+c_\pm z}\ge
	-\frac{3\varphi}{2}\ge-\frac{3\varphi_0}{2}
	+\frac{\pi}{2}+\eps,
\]
if $\sigma>0$, and that
\[
        \frac{\pi-\varphi_0}{2}
	\ge \arg \sqrt{\lambda +\kappa +c_\pm z}\ge
	-\frac{\varphi}{2}.
\]
By these two estimates we see that in any case we obtain 
\[
        \frac{3(\pi-\varphi_0)}{2}
	\ge \arg f_1(\lambda,z,\del,\sigma,\kappa)\ge
	-\frac{\varphi}{2}.
\]
and 
\[
        \pi-\varphi_0
	\ge f_2(\lambda,z,\del,\sigma,\kappa)\ge
	-\frac{3\varphi}{2}\ge-\frac{3\varphi_0}{2}
	+\frac{\pi}{2}+\eps.
\]
Consequently,
\[
	\frac{2\pi}{3}
        \ge \pi-\varphi_0 +\frac{\varphi}{2}
	\ge \arg \frac{f_2(\lambda,z,\del,\sigma,\kappa)}
	{f_1(\lambda,z,\del,\sigma,\kappa)}
	\ge -\frac{3\varphi_0}{2}
	+\frac{\pi}{2}+\eps
        -\frac{3(\pi-\varphi_0)}{2}
	=-\pi+\eps.
\]
A similar argument holds for the case that 
$(\lambda,z,\del,\sigma,\kappa)
\in\Sigma_{\pi-\varphi_0}\times\Sigma_{\varphi}
\times[0,\del^*]\times[0,\sigma^*]\times[1,\infty)$ 
with $\arg \lambda\le 0$. Here we obtain
\[
        -\frac{2\pi}{3}
	\le \arg \frac{f_2(\lambda,z,\del,\sigma,\kappa)}
	{f_1(\lambda,z,\del,\sigma,\kappa)}
	\le \pi-\eps.
\]
This implies that 
\[
	-1\ \not\in\ \overline{\Sigma}_{\pi-\eps}
	\supseteq\overline{\frac{f_2}{f_1} 
	\left(\Sigma_{\pi-\varphi_0}\times\Sigma_{\varphi}
	\times[0,\del^*]\times[0,\sigma^*]\times[1,\infty)\right)}
\]
Lemma~\ref{triangle} now yields the existence of a $c_1>0$
such that
\[
	|f_1(\lambda,z,\del,\sigma,\kappa)
	+f_2(\lambda,z,\del,\sigma,\kappa)|
	\ge c_1(|f_1(\lambda,z,\del,\sigma,\kappa)|
	+|f_2(\lambda,z,\del,\sigma,\kappa)|)
\]
for all $(\lambda,z,\del,\sigma,\kappa)
\in\Sigma_{\pi-\varphi_0}\times\Sigma_{\varphi}
\times[0,\del^*]\times[0,\sigma^*]\times[1,\infty)$.
An iterative application of Lemma~\ref{triangle} on the 
summands of $f_1$ and $f_2$ and an application 
of inequality (\ref{est_omega}) then
result in
\begin{eqnarray}
	&&|f(\lambda,z,\sigma,\del,\kappa)|\nonumber\\
	&&\ \ge \ c_2\left\{|\lambda|+\kappa+\sigma |z|
	      \left(\sqrt{|\lambda|}+\sqrt{\kappa}+\sqrt{c_+|z|}
	      +\sqrt{c_-|z|}\right)\right.\nonumber\\ 
	&&\ \ \ \      \strut +\del (|\lambda|+\kappa)
	      \left(\sqrt{|\lambda|}+\sqrt{\kappa}+\sqrt{c_+|z|}
	      +\sqrt{c_-|z|}\right)\nonumber\\
	&&\ \ \ \   \left.\strut + a_+\left(
	      \sqrt{|\lambda|}+\sqrt{\kappa}+\sqrt{c_+|z|}\right)
	      +a_-\left(
	      \sqrt{|\lambda|}+\sqrt{\kappa}+\sqrt{c_-|z|}\right)
	      \right\},\nonumber\quad\qquad
\end{eqnarray}
for all $(\lambda,z,\del,\sigma,\kappa)
\in\Sigma_{\pi-\varphi_0}\times\Sigma_{\varphi}
\times[0,\del^*]\times[0,\sigma^*]\times[1,\infty)$.
This implies that the functions
\[
	m_0:=\frac1{f},\ \ 
	m_1:=\frac{\lambda+\kappa}{f},\ \ 
	m_2:=\frac{\sqrt{z}}{f},\ \
	m_3:=\frac{\sigma z\sqrt{\lambda+\kappa}}{f},
\]
\[
	m_4:=\frac{\sigma z^{3/2}}{f},\ \
	m_5:=\frac{\del(\lambda+\kappa)^{3/2}}{f},\ \
	m_6:=\frac{\del(\lambda+\kappa) \sqrt{z}}{f}
\]
are uniformly bounded on 
$\Sigma_{\pi-\varphi_0}\times\Sigma_{\varphi}
\times[0,\del^*]\times[0,\sigma^*]\times[1,\infty)$.

Now consider the cases $R\ge\del\ge\del^*>0$ 
or $R\ge\sigma\ge\sigma^*>0$. We set
\[
	g(\lambda,z,\del,\sigma,\kappa)
	:=f(\lambda,z,\del,\sigma,\kappa)
	-a_+(\del,\sigma)\sqrt{c_+}\omega_+(\lambda,z)
	-a_-(\del,\sigma)\sqrt{c_-}\omega_-(\lambda,z).
\]
The argumentation above shows that 
\[
	\frac1{g},\ \ 
	\frac{\lambda+\kappa}{g},\ \ 
	\frac{\sigma z\sqrt{\lambda+\kappa}}{g},\ \ 
	\frac{\sigma z^{3/2}}{g},\ \
	\frac{\del(\lambda+\kappa)^{3/2}}{g},\ \
	\frac{\del(\lambda+\kappa) \sqrt{z}}{g}
\]
are still uniformly bounded functions and this even on
$\Sigma_{\pi-\varphi_0}\times\Sigma_{\varphi}
\times[0,R]^2\times[1,\infty)$. The aim now is to show
that the term 
$a_+(\del,\sigma)\sqrt{c_+}\omega_+(\lambda,z)
	+a_-(\del,\sigma)\sqrt{c_-}\omega_-(\lambda,z)$
can be regarded as a perturbation of $g$, if $\kappa$
is assumed to be large enough. Indeed, if $\del\ge\del^*>0$,
by using \eqref{a_smaller_m} we can estimate
\begin{eqnarray*}
	\left|\frac{a_\pm(\del,\sigma)\sqrt{c_\pm}\omega_\pm}
	{g}\right|
	&\le& \frac{CM}{\del^*|\lambda+\kappa|}
	      \left|\frac{\del(\lambda+\kappa)\omega_\pm}
	      {g}\right|\\
	&\le& \frac{C}{|\lambda|+\kappa}\quad
	      \le \ \frac{C}{\kappa}      
\end{eqnarray*}
for $(\lambda,z,\del,\sigma,\kappa)
\in\Sigma_{\pi-\varphi_0}\times\Sigma_{\varphi}
\times[\del^*,R]\times[0,R]\times[1,\infty)$. 
On the other hand, if $\sigma\ge\sigma^*>0$, we deduce
by virtue of \eqref{est_omega} that
\begin{eqnarray*}
	\left|\frac{a_\pm(\del,\sigma)\sqrt{c_\pm}\omega_\pm}
	{g}\right|
	&\le& \frac{CM}{|\omega_\pm|}
	      \left|\frac{\lambda+\kappa+c_\pm z}
	      {g}\right|\\
	&\le& \frac{C}{\sqrt{\kappa}}     
	      \left(\left|\frac{\lambda+\kappa}
	      {g}\right|
	      +\frac{1}{\sigma^*|\sqrt{\lambda+\kappa}|}
	      \left|\frac{\sigma z\sqrt{\lambda+\kappa}}
	      {g}\right|\right)\\
	&\le& \frac{C}{\sqrt{\kappa}}
\end{eqnarray*}
for $(\lambda,z,\del,\sigma,\kappa)
\in\Sigma_{\pi-\varphi_0}\times\Sigma_{\varphi}
\times[0,R]\times[\sigma^*,R]\times[1,\infty)$. 
Hence, for fixed $\kappa$ chosen large enough we see that
we can achieve 
\[
	\left|\frac{a_+(\del,\sigma)\sqrt{c_+}\omega_+
	+a_-(\del,\sigma)\sqrt{c_-}\omega_-}
	{g}\right|
	\le \frac12
\]
to be valid for $(\lambda,z,\del,\sigma)
\in\Sigma_{\pi-\varphi_0}\times\Sigma_{\varphi}
\times[\del^*,R]\times[0,R]$
or $(\lambda,z,\del,\sigma)
\in\Sigma_{\pi-\varphi_0}\times\Sigma_{\varphi}
\times[0,R]\times[\sigma^*,R]$.
Thus, we may represent $1/f$ as
\[
	\frac1{f}
	=\frac1{g}
	\left(1+\frac{a_+(\del,\sigma)\sqrt{c_+}\omega_+
	+a_-(\del,\sigma)\sqrt{c_-}\omega_-}{g}\right)^{-1},
\]
and therefore the functions $m_0,\ldots,m_6$
are uniformly bounded for all
$(\lambda,z,\del,\sigma)
\in\Sigma_{\pi-\varphi_0}\times\Sigma_{\varphi}\times[0,R]^2$.

The remaining argumentation is now analogous to the proof
of Lemma~\ref{dom_f}.
Employing (\ref{rh_infty_dn}) we obtain
\[
        \RR\biggl(\biggl\{
	\|m_j(\lambda,D_n,\del,\sigma)
	\|_{\LL({_0\F}^r_p(\R_+,\K^s_p(\R^n)))}
	:(\lambda,\del,\sigma)\in \Sigma_{\pi-\varphi_0}
	\times[0,R]^2\biggr\}\biggr)
        \le C,
\]
for $j=0,1,\ldots,6$. Consequently,
\[
	\|m_j(G,D_n,\del,\sigma)
	\|_{\LL({_0\F}^r_p(\R_+,\K^s_p(\R^n)))}
	\le C\quad ((\del,\sigma)\in[0,R]^2),
\]
by virtue of (\ref{rh_infty_g}) and \cite[Theorem~4.4]{KW01}.
The invertibility of the operators
\begin{eqnarray*}
	(G+1)^{1/2}&:&{_0\F}^{r+1/2}_p(\R_+,\K^s_p(\R^n))
	\to {_0\F}^r_p(\R_+,\K^s_p(\R^n)),\\
	D_n^{1/2}+1&:&{_0\F}^r_p(\R_+,\K^{s+1}_p(\R^n))
	\to {_0\F}^r_p(\R_+,\K^s_p(\R^n)),
\end{eqnarray*}
(see for instance Proposition 2.9 and Lemma 3.1 in \cite{MS11})
then yields the assertion, 
since $L^{-1}=m_0(G,D_n,\sigma,\del)$, and by employing
the fact that $h\mapsto h(G)$ is an algebra homomorphism
from $H^\infty(\Sigma_{\pi-\varphi_0},\KK_G(X))$
into $\LL(X)$ for  $X={_0\F}^r_p(\R_+,\K^s_p(\R^n))$ and
where
\[
	\KK_G(X):=\{B\in \LL(X):B(\mu-G)^{-1}=(\mu-G)^{-1}B,
	\ \mu\in\rho(G)\}.
\]
\end{proof}
\noindent
(iv) 
We turn to the proof of the corresponding regularity assertions in 
Theorem~\ref{homog_ini} for $(u,\eta,\eta_E)$. 
According to the results in \cite[pages 15--16]{EPS03},
\begin{equation}
\label{iff}
\begin{split}
&\int_0^\infty \e^{-F_+s/\sqrt{c_+}}f^+(s)\dd s 
\in {_0\FF}^3_\infty
\iff f^+\in L_p(\R_+,L_p(\R^{n+1}_+)).
\end{split}
\end{equation}
By the same arguments we also have
\begin{equation}
\label{iff1}
\begin{split}
&\int_0^\infty \e^{-F_-s/\sqrt{c_-}}f^-(-s)\dd s
\in {_0\FF}^3_\infty
\iff f^-\in L_p(\R_+,L_p(\R^{n+1}_-)).
\end{split}
\end{equation}
Next, note that by Lemma~\ref{dom_f} we have that 
\begin{equation}
\label{F^2toF^3}
F_\pm\in{\rm Isom}({_0\FF_\infty^2},{_0\FF_\infty^3}).
\end{equation}
Indeed, we obtain
\begin{eqnarray*}
	F^{-1}_\pm({_0\FF^3_\infty})
	&=&{_0W}^{1-1/2p}_p(\R_+,L_p(\R^n))
	   \cap {_0W}^{1/2-1/2p}_p(\R_+, W^1_p(\R^n))\\
	&& \cap \ {_0H}^{1/2}_p(\R_+, W^{1-1/p}_p(\R^n))
	   \cap L_p(\R_+, W^{2-1/p}_p(\R^n))\\
	&=&{_0\FF}^2_\infty,
\end{eqnarray*}
by virtue of the embedding
\[
	{_0\FF}^2_\infty\hookrightarrow
	{_0W}^{1/2-1/2p}_p(\R_+,W^1_p(\R^n))
	\cap {_0H}^{1/2}_p(\R_+,W^{1-1/p}_p(\R^n)),
\]
which is a consequence of the mixed derivative theorem.
Thus all the terms inside the brackets on the right hand side of
(\ref{form_rho}) belong to the space ${_0\FF_\infty^3}$.
In the same way as we clarified the invertibility of
$F_\pm:{_0\FF_\infty^2}\to {_0\FF_\infty^3}$ by applying
Lemma~\ref{dom_f}, we can see that 
$L:{_0\EE_\infty^2(\del,\sigma)}\to{_0\FF_\infty^3}$ is invertible
by an application of Proposition~\ref{reg_op_l}. For instance, 
if $\del,\sigma>0$, this
follows from the embedding
\[
	{_0\EE}^2_\infty(\del,\sigma)\hookrightarrow
	{_0H}^{3/2}_p(\R_+, W^{1-1/p}_p(\R^n))
	\cap {_0 W^{1-1/2p}_p(\R_+, W^{3}_p(\R^n))},
\]
which is again a consequence of the mixed derivative theorem.
Furthermore, Proposition~\ref{reg_op_l} implies the estimate
\[
	\|L^{-1}\|_{\LL({_0\FF_\infty^3},{_0\EE_\infty^2(0,0)})}
	+\del\|L^{-1}\|_{\LL({_0\FF_\infty^3},
	{_0\EE_\infty^2(1,0)})}
	+\sigma\|L^{-1}\|_{\LL({_0\FF_\infty^3},
	{_0\EE_\infty^2(0,1)})}
	\le C
\]
for $0\le\del,\sigma\le R$.
Altogether this gives us
\begin{equation}\label{est_eta}
	\|\eta\|_{_0\EE_\infty^2(\del,\sigma)}
	\le C\left(\|f\|_{\FF_\infty^1}
	+\|g\|_{_0\FF_\infty^2}
	+\|h\|_{_0\FF_\infty^3}\right)
\end{equation}
for $(\del,\sigma)\in[0,R]^2$,
which yields the desired regularity for $\eta$.
Observe that $u$ now can be regarded as the solution of the 
diffusion equation 
\[
    \left\{
        \begin{array}{r@{\quad=\quad}ll}
        (\partial_t+\kappa-c\Delta)u 
                &f
        & \mbox{in}\ (0,\infty)\times\dR^{n+1}\,\\
        \gamma u^{\pm} & g+\sigma\Delta_x\eta
	-\del(\partial_t+\kappa)\eta 
        & \mbox{on}\ (0,\infty)\times\R^n,\\
        u(0)  & 0
        & \mbox{in}\ \dR^{n+1}.\\
        \end{array}
        \right.
\]
A trivial but important observation now is that this
equation itself does not depend on $\del$ and $\sigma$,
but only the data. Therefore also the corresponding
solution operator is independent of $\del$ and $\sigma$.
By well-known results (see e.g. \cite[Proposition 5.1]{EPS03}) 
and in view of (\ref{est_eta}) we obtain
\begin{eqnarray*}
	\|u\|_{_0\EE_\infty^1}
	&\le& C\left(\|f\|_{\FF_\infty^1}
	      +\|g\|_{_0\FF_\infty^2}
	      +\del\|\eta\|_{_0\EE_\infty^2(1,0)}
	      +\sigma\|\eta\|_{_0\EE_\infty^2(0,1)}\right)\\
	&\le& C\|(f,g,h)\|_{_0\FF_\infty}
	\qquad (0\le\del,\sigma\le R).
\end{eqnarray*}
Similarly we can proceed for $\eta_E$. Since it satisfies
equation (\ref{rho_E-infty}), we deduce
\[
	\|\eta_E\|_{_0\EE_\infty^1}
	\le C\|\eta\|_{_0\FF_\infty^2}.
\]
By virtue of ${_0\EE_\infty^2(0,0)}\hookrightarrow{_0\FF_\infty^2}$
and again (\ref{est_eta}) we conclude that
\[
	\|\eta_E\|_{_0\EE_\infty^1}
	\le C\|(f,g,h)\|_{_0\FF_\infty}
	\qquad(0\le\del,\sigma\le R).
\]
\smallskip\\
(v) 
Let $T_0>0$ be fixed, and let
$J:=(0,T)$ with $T\le T_0.$ 
We set
\begin{equation}\label{def_retraction}
\begin{aligned}
{\mathcal R}^c_{\!J}&:\ {_0\FF_T}\ \to \ {_0\FF_\infty}, \\
(f,g,h)&\ \mapsto \ 
(e^{-\kappa t}(\E_{\!J}f),e^{-\kappa t}(\E_{\!J}g),
e^{-\kappa t}(\E_{\!J}h)),
\end{aligned}
\end{equation}
where $\E_{\!J}$ is defined as
\[
        \E_{\!J} u(t):=\E_{\!J,r}u(t):=
        \left\{\begin{array}{ll}
              u(t)   & \mbox{if}\quad 0\le t\le T,\\
                  u(2T-t)& \mbox{if}\quad T\le t\le 2T, \\
                  0      & \mbox{if}\quad 2T\le t.\\
                  \end{array}
                  \right.
\]
It follows from \cite[Proposition~6.1]{PSS07}
and the fact
\[
	\|(e^{-\kappa t}(\E_{\!J}f),(e^{-\kappa t}(\E_{\!J}g),
	e^{-\kappa t}(\E_{\!J}h))\|_{_0\FF_\infty}
	\le \|e^{-\kappa t}\|_{\BUC^1(\R_+)}\|(\E_{\!J}f,
	\E_{\!J}g,\E_{\!J}h)\|_{_0\FF_\infty}
\]
that there
exists a positive constant $c_0=c_0(T_0)$ such that
\begin{equation}
\label{exp-extension}
\|{\mathcal R}^c_{\!J}(f,g,h)\|_{_0\FF_\infty}
\le c_0 \|(f,g,h)\|_{_0\FF_T}\qquad ((f,g,h)\in {_0\FF_T})
\end{equation}
for any interval $J=(0,T)$ with $T\le T_0$.
\smallskip\\
Let $(u,\eta,{\eta_E})\in{_0\EE_\infty(\del,\sigma)}$
be the solution of \eqref{LOS-infty}--\eqref{rho_E-infty},
with $(f,g,h)$ replaced by $({\mathcal R}^c_{\!J}(f,g,h))$,
whose existence has been established in steps
(i)--(iv) of the proof. We note that 
\begin{eqnarray*}
\label{schranke}
	\|(u,\eta,\eta_E)\|_{_0\EE_\infty(\del,\sigma)}
	&\le& K \|{\mathcal R}^c_{\!J}(f,g,h)\|_{_0\FF_\infty}\\
	&\le& Kc_0 \|(f,g,h)\|_{_0\FF_T}
\end{eqnarray*}
for any $(f,g,h)\in{_0\FF_T}$, $0\le\del,\sigma\le R$, 
and any interval $J=(0,T)$ with $T\le T_0$,
where $K$ is a universal constant. 
Now, let 
$$(v,\rho,\rho_E):=
({\mathcal R_{\!J}}(e^{\kappa t}u), 
{\mathcal R_{\!J}}(e^{\kappa t}\eta),
{\mathcal R_{\!J}}(e^{\kappa t}\eta_E))$$
where ${\mathcal R_{\!J}}$ denotes the restriction operator,
defined by ${\mathcal R_{\!J}}w:=w|_J$ for $w:\R_+\to X$. 
Then it is easy to verify that
\begin{equation}
\label{irgendwas}
	(v,\rho,\rho_E)\in{_0\EE_T(\del,\sigma)},
	\quad\text{$(v,\rho,\rho_E)$ 
	solves \eqref{LTS}--\eqref{2-rho_E}}
\end{equation}
and that there is a constant $M=M(T_0)$
such that
\begin{equation*}
	\|(v,\rho,\rho_E)\|_{_0\EE_T(\del,\sigma)}
	\le M \|(f,g,h)\|_{_0\FF_T}
\end{equation*}
for $0\le\del,\sigma\le R$, and $T\le T_0$.
Finally, uniqueness follows by a direct calculation
which is straight forward and therefore omitted here.
This completes the proof.
\end{proof}
We proceed with convergence results for the case of 
zero time traces. To indicate the dependence on
the parameters $\del$ and $\sigma$ we label from now on
the corresponding functions and operators by $\mu$,
as e.g. $L_\mu$, $v^\mu$,  where $\mu=(\del,\sigma)$.
\begin{corollary}
\label{conv_op_l}
Let $1<p<\infty$, $R>0$, $0\le\del_0\le\del\le R$,
and $0\le\sigma_0\le\sigma\le R$. Suppose that $a$ is a function 
satisfying the conditions in \eqref{cont_cond_a}, and 
let $L_\mu$ be the operator
defined in \eqref{L} corresponding to the parameter
$\mu:=(\del,\sigma)$.
Then we have 
\begin{equation}\label{conv_l_ku}
	(\del-\del_0)L_\mu^{-1}\to 0
	\quad\mbox{stronly in}\quad
	\LL({_0\FF^3_\infty},{_0\EE^2_\infty(1,0)}),
\end{equation}
\begin{equation}\label{conv_l_st}
	(\sigma-\sigma_0)L_\mu^{-1}\to 0
	\quad\mbox{stronly in}\quad
	\LL({_0\FF^3_\infty},{_0\EE^2_\infty(0,1)}),
\end{equation}
and
\begin{equation}\label{conv_l_3}
	L_\mu^{-1}\to L_{\mu_0}^{-1}
	\quad\mbox{stronly in}\quad
	\LL({_0\FF^3_\infty},{_0\EE^2_\infty(\mu_0)}),
\end{equation}
as $\mu\to \mu_0$, where $\mu_0=(\del_0,\sigma_0)$.
\end{corollary}
\goodbreak
\begin{proof}
As pointed out in part (iv) of the proof of Theorem~\ref{homog_ini}
the domain of the operator
$F_+$ in ${_0\FF^3_\infty}$ is ${_0\FF^2_\infty}$.
This implies that
\[
	\DD(F_+^3)\hookrightarrow
	{_0W}^{2-1/2p}_p(\R_+,L_p(\R^n))
	\cap L_p(\R_+,W^{4-1/p}_p(\R^n))
	={_0\EE^2_\infty(1,1)}.
\]
Now pick $f\in \DD(F_+^3)$. From Proposition~\ref{reg_op_l} we infer
that 
\begin{equation}
\label{uniform-00}
	\|L_\mu^{-1}\|_{\LL({_0\FF^3_\infty},{_0\EE^2_\infty(0,0)})}
	\le C\qquad (\mu\in[0,R]^2).
\end{equation}
This yields
\begin{eqnarray*}
	\|(\del-\del_0)L_\mu^{-1}f\|_{{_0\EE^2_\infty(1,0)}}
	&\le&C(\del-\del_0)\left(\|(G+\kappa)^{3/2}L_\mu^{-1}
	     f\|_{W^{1/2-1/2p}_p(\R_+,L_p(\R^n))}\right.\\ 
        &&   \left.\strut  +\|(G+\kappa)L_\mu^{-1}f
	     \|_{L_p(\R_+,W^{2-1/p}_p(\R^n))} 
	     \right)\\
	&\le& C(\del-\del_0)\|f\|_{_0\EE^2_\infty(1,1)}\\
	&\to& 0\qquad (\mu\to\mu_0).
\end{eqnarray*}
Since $\DD(F_+^3)$ is dense in $_0\FF^3_\infty$, 
the uniform boundedness of
 $\|L_\mu^{-1}\|_{\LL({_0\FF^3_\infty},{_0\EE^2_\infty(\mu)})}$ in $\mu\in[0,R]^2$
(which yields uniform boundedness of 
$(\delta-\delta_0)\|L_\mu^{-1}\|_{\LL({_0\FF^3_\infty},{_0\EE^2_\infty(1,0)})}$
for $\del\in [\del_0,R],\; \sigma\in [0,R]$)
implies \eqref{conv_l_ku}.
In a very similar way \eqref{conv_l_st} can be proved.
In order to see \eqref{conv_l_3} we write
\begin{eqnarray*}
	L_{\mu_0}^{-1}-L_{\mu}^{-1}
	&=&L_{\mu_0}^{-1}\left(L_{\mu}-L_{\mu_0} \right)L_{\mu}^{-1}\\
	&=&L_{\mu_0}^{-1}\left\{\left((\del-\del_0)(G+\kappa)+(\sigma-\sigma_0)D_n\right)
	   \left(\sqrt{c_+}F_+ +\sqrt{c_-}F_-\right)\right\} L_{\mu}^{-1}\\
	&+& L_{\mu_0}^{-1}\left\{(a_{+}(\mu)-a_{+}(\mu_0))\sqrt{c_{+}}F_{+}+ 
	                       (a_{-}(\mu)-a_{-}(\mu_0))\sqrt{c_{-}}F_{-} \right\} L_{\mu}^{-1}.    
\end{eqnarray*}
In view of 
$L_{\mu_0}^{-1}\in \LL({_0\FF^3_\infty},{_0\EE^2_\infty(\mu_0)})$
 this representation shows that 
\eqref{conv_l_3} is obtained as a consequence of 
\eqref{conv_l_ku}-\eqref{conv_l_st},
and \eqref{uniform-00} in conjunction with the continuity of $a_\pm$.
\end{proof}
Based on this result we will now prove
convergence of solutions of problem \eqref{LTS}-\eqref{2-rho_E}.
\begin{theorem}
\label{conv_homog_ini}
Let $3<p<\infty$, $R,T>0$, $0\le\del_0\le\del\le R$,
and $0\le\sigma_0\le\sigma\le R$. Suppose that $a$ is a function 
satisfying the conditions in \eqref{cont_cond_a} and that
\[
	((f^\mu,g^\mu,h^\mu))_{\mu\in[\del_0,R]
	\times[\sigma_0,R]}\subseteq {_0\FF_T}.
\]
Furthermore, denote by $(v^\mu,\rho^\mu,\rho_E^\mu)$ the 
unique solution
of \eqref{LTS}-\eqref{2-rho_E} whose existence is established
in Theorem~\ref{homog_ini} and that corresponds to 
the parameter $\mu=(\del,\sigma)$.
Then, if 
\begin{equation}\label{conv_homog_data}
	(f^\mu,g^\mu,h^\mu)
        \to (f^{\mu_0},g^{\mu_0},h^{\mu_0})
	\quad\mbox{in}\quad {_0\FF_T}
	\qquad (\mu\to\mu_0),
\end{equation}
we have that
\begin{equation}\label{conv_sol_homog_ini}
	(v^\mu,\rho^\mu,\rho_E^\mu)
	\to(v^{\mu_0},\rho^{\mu_0},\rho_E^{\mu_0})
	\quad\mbox{in}\quad {_0\EE_T(\mu_0)}
	\qquad (\mu\to\mu_0),
\end{equation}
where $\mu_0=(\del_0,\sigma_0)$. In particular,
if 
\[
	S_\mu^{-1}:(f,g,h)\mapsto
	(v^\mu,\rho^\mu,\rho_E^\mu)
\]
denotes the solution operator to system
\eqref{LTS}, we have that
\begin{equation}\label{conv_sol_op_s}
	S_\mu^{-1}\to S_{\mu_0}^{-1}
	\quad\mbox{strongly in}\quad \LL({_0\FF_T},
	{_0\EE^1_T}\times{_0\EE^2_T(\mu_0)}\times {_0\EE^1_T})
	\qquad (\mu\to\mu_0).
\end{equation}
\end{theorem}
\begin{proof}
In view of the arguments in part (v) of the proof of 
Theorem~\ref{homog_ini} the solution 
$(v^\mu,\rho^\mu,\rho_E^\mu)$ can be represented by
\begin{equation}\label{rep_64}
	(v^\mu,\rho^\mu,\rho_E^\mu):=
	({\mathcal R_{\!J}}(e^{\kappa t}u^\mu), 
	{\mathcal R_{\!J}}(e^{\kappa t}\eta^\mu),
	{\mathcal R_{\!J}}(e^{\kappa t}\eta_E^\mu)),
\end{equation}
where ${\mathcal R_{\!J}}$ denotes the restriction operator
and $(u^\mu,\eta^\mu,\eta_E^\mu)$ is the solution of 
\eqref{LOS-infty}--\eqref{rho_E-infty} with right hand side
$({\mathcal R}^c_{\!J}(f^\mu,g^\mu,h^\mu))$ 
and ${\mathcal R}^c_{\!J}$
as defined in \eqref{def_retraction}. Hence we see 
that it suffices to prove convergence for the vector
$(u^\mu,\eta^\mu,\eta_E^\mu)$.
Clearly, \eqref{conv_homog_data} implies
that 
\[	
	({\mathcal R}^c_{\!J}(f^\mu,g^\mu,h^\mu))
        \to ({\mathcal R}^c_{\!J}(f^{\mu_0},g^{\mu_0},h^{\mu_0}))
	\quad\mbox{in}\quad {_0\FF_\infty}
	\qquad (\mu\to\mu_0).
\]
Therefore, and for simplicity, we simlpy write 
$(f^\mu,g^\mu,h^\mu)$
for the data instead of 
$({\mathcal R}^c_{\!J}(f^\mu,g^\mu,h^\mu))$
in the remaining part of the proof.

Next, recall from \eqref{form_rho}
that $\eta^\mu$ is given by
\[
	\eta^\mu=L_\mu^{-1}\ell^\mu
	\qquad
	(\mu\in[\del_0,\infty)\times[\sigma_0,\infty))
\]
with
\begin{equation*}
\begin{split}
        \ell^\mu=&
             h^\mu-\int_0^\infty \e^{-F_+ s/\sqrt{c_+}}
             (f^\mu)^+(s)\dd s
             -\int_0^\infty \e^{-F_- s/\sqrt{c_-}}
             (f^\mu)^-(-s)\dd s\\
	     &\strut + \sqrt{c_+}F_+ g^\mu
	     +\sqrt{c_-}F_- g^\mu.
\end{split}
\end{equation*}
According to \eqref{F^2toF^3} we know that 
$F_{\pm}\in {\rm Isom}({_0\FF^2_\infty},{_0\FF^3_\infty})$. This fact
and relations \eqref{iff} and \eqref{iff1} then imply, by virtue of
assumption \eqref{conv_homog_data}, that
\begin{eqnarray*}
	\|\ell^\mu-\ell^{\mu_0}\|_{_0\FF^3_\infty}
	&\le& C\left(
	      \|f^\mu-f^{\mu_0}\|_{\FF^1_T}
	      +\|g^\mu-g^{\mu_0}\|_{_0\FF^2_T}
	      +\|h^\mu-h^{\mu_0}\|_{_0\FF^3_T}
	      \right)\\
	&\to& 0\qquad\qquad (\mu\to\mu_0).
\end{eqnarray*}
By the uniform boundedness of 
$\|L_\mu^{-1}\|_{\LL({_0\FF^3_T},{_0\EE^2_T(\mu)})}$ 
in $\mu\in[\del_0,R]\times[\sigma_0,R]$ 
(see Proposition~\ref{reg_op_l}) and because
$(\del-\del_0)L_\mu^{-1}\to 0$ strongly in
$\LL({_0\FF^3_T},{_0\EE^2_T(1,0)})$ and
$(\sigma-\sigma_0)L_\mu^{-1}\to 0$ strongly in
$\LL({_0\FF^3_T},{_0\EE^2_T(0,1)})$ 
(see Corollary~\ref{conv_op_l}) this results in
\begin{equation}\label{del_eta_to_0}
	(\del-\del_0)\eta^\mu\to 0
	\quad\mbox{in}\quad {_0\EE^2_\infty(1,0)}
\end{equation}
and
\begin{equation}\label{sigma_eta_to_0}
	(\sigma-\sigma_0)\eta^\mu\to 0
	\quad\mbox{in}\quad {_0\EE^2_\infty(0,1)}.
\end{equation}
Now, denote by
\[
        U_\mu: (u^\mu,\eta^\mu)\mapsto
        (f^\mu,g^\mu,h^\mu)
\]
the operator that maps the solution to the data
corresponding to system \eqref{LOS-infty}.
From part (iv) of the proof of 
Theorem~\ref{homog_ini} we infer that
\begin{equation}\label{isom_u}
        U_\mu\in {\rm Isom}({_0\EE^1_\infty}\times
	{_0\EE^2_\infty(\mu)},{_0\FF_\infty})
        \qquad (\mu\in [\del_0,R]\times[\sigma_0,R]).
\end{equation}
Furthermore, observe that we have
\begin{eqnarray*}
        &&(u^\mu,\eta^\mu)-(u^{\mu_0},\eta^{\mu_0})\\
        &&=U_{\mu}^{-1}(f^\mu,g^\mu,h^\mu)
           -U_{\mu_0}^{-1}(f^{\mu_0},g^{\mu_0},h^{\mu_0})\\
        &&=U_{\mu_0}^{-1}
        \left(
        \begin{array}{c}
        f^\mu-f^{\mu_0}\\
        g^\mu-g^{\mu_0}
        +(\sigma-\sigma_0)\Delta_x\eta^\mu
        -(\del-\del_0)(\partial_t+\kappa)\eta^\mu\\
        h^\mu-h^{\mu_0}
	+[\![c\gamma\partial_y(a(\mu)-a(\mu_0))\eta^\mu_E ]\!]
        \end{array}
        \right)^T.
\end{eqnarray*}
Relation \eqref{isom_u} applied for $\mu=\mu_0$ then yields
\begin{eqnarray*}
	&&\|(u^\mu,\eta^\mu)-(u^{\mu_0},\eta^{\mu_0})
	\|_{{_0\EE^1_\infty}\times{_0\EE^2_\infty}(\mu_0)}\\
        && \le C\biggl(\|(f^\mu,g^\mu,h^\mu)
           -(f^{\mu_0},g^{\mu_0},h^{\mu_0})\|_{{_0\FF_\infty}}
            +(\del-\del_0)\|\eta^\mu\|_{_0\EE^2_\infty(1,0)}
	   \biggr.
           \\
        && \biggl.\quad\strut
            +(\sigma-\sigma_0)
	    \|\eta^{\mu}\|_{_0\EE^2_\infty(0,1)}
            +|a(\mu)-a(\mu_0)|\|\eta^\mu_E\|_{_0\EE^1_T}
	    \biggr).
\end{eqnarray*}
From Theorem~\ref{homog_ini} we know that
$\|\eta^\mu_E\|_{_0\EE^1_T}$ is uniformly bounded
in $\mu\in I_0$.
Thus, by \eqref{cont_cond_a}, 
\eqref{del_eta_to_0}, \eqref{sigma_eta_to_0},
and assumption \eqref{conv_homog_data} we conclude that 
\[
	(u^\mu,\eta^\mu)
	\to(u^{\mu_0},\eta^{\mu_0})
	\quad\mbox{in}\quad {_0\EE^1_\infty}
	\times{_0\EE^2_\infty(\mu_0)}
	\qquad (\mu\to\mu_0).
\]
The convergence of $\eta_{E}^\mu$ is easily obtained as a consequence
of the convergence of $\eta^\mu$.
Recall that $\eta_{E}^\mu$ is the solution of 
\eqref{2-rho_E} with $\rho$ replaced by $\eta^\mu$. Denote by 
$\T$ the solution operator of this diffusion equation which 
is obviously independent of $\mu$. Then 
by \cite[~Proposition 5.1]{EPS03} we obtain 
\begin{eqnarray}
	\|\eta_{E}^\mu-\rho_{E}^{\mu_0}\|_{_0\EE^1_\infty}
	&=&\|\T(0,\eta^\mu-\eta^{\mu_0},0)\|_{_0\EE^1_\infty}
	      \nonumber\\
	&\le& C
	      \|\eta^\mu-\eta^{\mu_0}\|_{_0\FF^2_\infty}
	      \nonumber\\
	&\le& C
	      \|\eta^\mu-\eta^{\mu_0}\|_{_0\EE^2_\infty(0,0)}
	      \nonumber\\
	&\to& 0\qquad\qquad (\mu\to\mu_0),
	\label{est_34}
\end{eqnarray}
by the just established convergence of $\eta^\mu$. 
Representation \eqref{rep_64} then implies 
\eqref{conv_sol_homog_ini}.

Obviously \eqref{conv_sol_homog_ini} is still true for
fixed data, i.e., if 
\[
	(f^\mu,g^\mu,h^\mu)
	=(f,g,h)\in {_0\FF_T}
	\qquad (\mu\in[\del_0,R]\times[\sigma_0,R]).
\]
Hence \eqref{conv_sol_op_s} readily follows from 
\eqref{conv_sol_homog_ini}.
\end{proof}
\subsection{Inhomogeneous time traces}\label{inhom_traces}
Next we consider the fully inhomogeneous 
system \eqref{LTS}--\eqref{2-rho_E} and we will prove
Theorem~\ref{main_lin}.
By introducing appropriate auxiliary functions,
we will reduce this problem
to the situation of Theorem~\ref{homog_ini}.
\vspace{4mm}
\begin{proof} (of Theorem~\ref{main_lin}.)  
If $\del=\sigma=0$ this result is proved in 
\cite[Theorem~3.4]{PSS07}\footnote{Actually with $g=0$.
But by obvious changes in the proof one can obtain
the result also for $0\neq g\in\FF^2_T$.}.
So, we may assume that $\del>0$ or $\sigma>0$ which implies
that $\rho_0\in\rcsisu$. Furthermore, it 
follows from the trace results in \cite{DPZ08}
that the conditions listed in 
\eqref{assum_ml_a}--\eqref{assum_ml_c} are necessary.

Suppose we had a solution $(v,\rho,\rho_E)$ 
of \eqref{LTS}--\eqref{2-rho_E} 
as claimed in the statement of Theorem~\ref{main_lin}.
Let $v_1$ be the solution of the two-phase 
diffusion equation
\begin{equation}
\label{diffusion}
    \left\{
        \begin{array}{r@{\quad=\quad}ll}
        (\partial_t-c\Delta )v_1 
                &f      & \mbox{in}\ J\times\dR^{n+1},\\
        \gamma v_1^{\pm} & g+\e^{-(1-\Delta_x)t}\zeta 
	& \mbox{on}\ J\times\R^n,\\
        v_1(0)  & v_0
        & \mbox{in}\ \dR^{n+1},\\
        \end{array}
        \right.
\end{equation}
with 
\begin{equation}
\label{def_zeta}
	\zeta:=\gamma v_0-g(0).
\end{equation}
Observe that by compatibility assumption~\eqref{assum_ml_b} we have 
\begin{equation}
\label{zeta}
	\zeta
	=(\sigma\Delta_x\rho-\del\partial_t\rho)|_{t=0}.
\end{equation}
Next let $\rho_1$ be an
extension function so that
\begin{equation}
\label{ABC}
\left(\rho_1(0),\partial_t\rho_1(0)\right):=
\left(\rho_0,h(0)-[\![c\gamma\partial_y(v_0
-a e^{-|y|(1-\Delta_x)^{1/2}}\rho_0)]\!]\right),
\end{equation}
as constructed in Lemma~\ref{ext_ext}, 
and let $\rho_{1,E}$
be the solution of \eqref{2-rho_E},
with $\rho$ replaced by $\rho_1$.
For the solvability of \eqref{diffusion} and the existence of
$\rho_1$ we have to check the required regularity and
compatibility conditions for the data.
By construction we have that $g(0)+\zeta=\gamma v_0$ 
and by the regularity assumptions on $g$ and $v_0$ we deduce
\[
	\zeta=\gamma v_0-g(0)\in W^{2-3/p}_p(\R^n),
\]
hence that
\begin{equation}\label{23}
	e^{-(1-\Delta_x)t}\zeta
	\in \FF^2_T.
\end{equation}
Then it follows from \cite[~Proposition 5.1]{EPS03}
that there is a unique solution $v_1\in\EE^1_T$ of
\eqref{diffusion}. 
Furthermore, if $\del>0$, we may use compatibility condition
\eqref{assum_ml_b} to obtain that
\[
	h(0)-[\![c\gamma\partial_y(v_0
	-a e^{-|y|(1-\Delta_x)^{1/2}}\rho_0)]\!]
	=\frac1{\del}(g(0)-\gamma v_0+\sigma\Delta_x\rho_0)
	\in W^{2-3/p}_p(\R^n).
\]
If $\del=0$, we may impose $\sigma>0$ which gives
\[
	c\gamma\partial_y
	a e^{-|y|(1-\Delta_x)^{1/2}}\rho_0
	=\mp ca(1-\Delta_x)^{1/2}\rho_0
	\in W^{3-3/p}_p(\R^n)\hookrightarrow 
	W^{2-6/p}_p(\R^n)
\]
in view of $\rho_0\in\rcsisu$.
Assumption~\eqref{assum_ml_c} then implies that
\[
	h(0)-[\![c\gamma\partial_y(v_0
	-a e^{-|y|(1-\Delta_x)^{1/2}}\rho_0)]\!]
	\in W^{2-6/p}_p(\R^n).
\]
Thus, in any case we can satisfy the assumptions of 
Lemma~\ref{ext_ext} which yields the existence of
$\rho_1\in\EE^2_T(\del,\sigma)$
as claimed, and of $\rho_{1,E}\in\EE^1_T$
by virtue of Remark~\ref{Bemerkungen}(b).

Now we set 
\[
	(v_2,\rho_2,\rho_{2,E})
	=(v,\rho,\rho_E)-
	 (v_1,\rho_1,\rho_{1,E}).
\]
It is clear that $\rho_{2,E}$ is the extension of $\rho_2$
given by \eqref{2-rho_E} with $\rho$ replaced by $\rho_2$.
Thus, $(v_2,\rho_2,\rho_{2,E})$ satisfies
\begin{equation}
\label{reduced}
        \left\{
        \begin{array}{r@{\quad=\quad}ll}
        (\partial_t-c\Delta)v_2
        & 0     & \mbox{in}\ J\times\dR^{n+1},\\
        \gamma v_2^\pm -\sigma\Delta_x\rho_2
	+\del\partial_t\rho_2
	&\sigma\Delta_x\rho_1-\del\partial_t\rho_1 
	-e^{-(1-\Delta_x)t}\zeta& \mbox{on}\ J\times \R^n, \\
        \partial_t\rho_2+[\![c\gamma\partial_y (v_2-a\rho_{2,E})]\!] 
        & h-\partial_t\rho_1 -[\![c\gamma\partial_y(v_1-a\rho_{1,E})]\!]      
        & \mbox{on}\ J\times\R^n, \\
        v_2(0) & 0 & \mbox{in}\ \dR^{n+1},\\
        \rho_2(0) & 0 &  \mbox{in}\ \R^n,\\
        \end{array}
        \right.
\end{equation}
and
\begin{equation}
\label{ABB}
    \left\{
        \begin{array}{r@{\quad=\quad}ll}
        (\partial_t-c\Delta)\rho_{2,E} 
                &0
        & \mbox{in}\ J\times\dR^{n+1},\\
        \gamma \rho_{2,E}^\pm & \rho_2 & \mbox{on}\ J\times\R^n,\\
          \rho_{2,E}(0)  & 0      & \mbox{in}\ \dR^{n+1}.\\
        \end{array}\qquad\quad
        \right.
\end{equation}
By construction, $\rho_1\in\EE^2_T(\del,\sigma)$, and by 
(\ref{23}) one may readily check that
\[
	\sigma\Delta_x\rho_1-\del\partial_t\rho_1 
	-e^{-(1-\Delta_x)t}\zeta\in {_0\FF_T^2}
\]
and that
\[
	h-\partial_t\rho_1 -[\![c\gamma\partial_y(v_1-a\rho_{1,E})]\!]
	\in{_0\FF^3_T}.
\]
Thus, by Theorem~\ref{homog_ini} the reduced system 
\eqref{reduced}--\eqref{ABB}
is uniquely solvable. 
This allows us to reverse the argument.  
In fact, since the solution $v_1$ of \eqref{diffusion}
and the extension $\rho_1$ depend on the data only,
the right hand side of 
\eqref{reduced}--\eqref{ABB} so does as well.
Theorem~\ref{homog_ini} now yields a unique solution 
$(v_2,\rho_2,\rho_{2,E})\in{_0\EE_T(\del,\sigma)}$
and 
\begin{equation}\label{splitting}
	(v,\rho,\rho_E)
	:=(v_2,\rho_2,\rho_{2,E})+
	 (v_1,\rho_1,\rho_{1,E})
\end{equation}
then solves the original system
\eqref{LTS}--\eqref{2-rho_E} in the reguarity classes required.
It remains to verify estimate~\eqref{max_reg_ineq_inhom}.
Observe that by Theorem~\ref{homog_ini} we know that
\begin{eqnarray*}
	&&\hspace{-3mm}\|(v_2,\rho_2,
	\rho_{2,E})\|_{_0\EE_T(\del,\sigma)}\\
	&&\hspace{-3mm}\le C\left(\|\sigma\Delta_x\rho_1
	-\del\partial_t\rho_1 
	-e^{-(1-\Delta_x)t}\zeta\|_{_0\FF^2_T}
	+\|h-\partial_t\rho_1 
	-[\![c\gamma\partial_y(v_1-a\rho_{1,E})]\!]\|_{_0\FF^3_T}\right)
\end{eqnarray*}
with $C>0$ independent of $\del,\sigma$. By 
$|a(\mu)|\le C$ for $\mu\in [0,R]^2$ and 
the facts pointed out 
above we can continue this calculation to the result
\begin{eqnarray}
	&&\hspace{-8mm}\|(v_2,\rho_2,\rho_{2,E})\|_{_0\EE_T(\del,\sigma)}
	\nonumber\\
	&&\hspace{-8mm}\le C\left(\sigma\|\Delta_x\rho_1\|_{\FF^2_T}
	+\del\|\partial_t\rho_1\|_{\FF^2_T} 
	+\|\zeta\|_{W^{2-3/p}_p(\R^n)}
	+\|h\|_{\FF^3_T}\right. 	\nonumber\\ 
	&&\hspace{-8mm} 	\left. \quad
	\strut+\|\partial_t\rho_1\|_{\FF^3_T} 
	+\|v_1-a\rho_{1,E}\|_{\EE^1_T}\right)\nonumber\\
	&&\hspace{-8mm}\le C\left(\|(v_1,\rho_1,
	\rho_{1,E})\|_{\EE_T(\del,\sigma)}
	+\|(0,g,h,v_0,0)\|_{\FF_T(0,0)}
	\right).
	\label{est_red_sol}
\end{eqnarray}
Hence we see that it remains to derive suitable estimates for 
$(v_1,\rho_1,\rho_{1,E})$.
Observe that equation \eqref{diffusion} does not
depend on $\del,\sigma$. 
By \cite[~Proposition 5.1]{EPS03} we deduce
\begin{eqnarray}
	\|v_1\|_{_0\EE^1_T}
	&\le& C\left(\|f\|_{\FF^1_T}
	+\|g+e^{-(1-\Delta_x)t}\zeta\|_{\FF^2_T} 
	+\|v_0\|_{\FF^4_T}\right)\nonumber\\
	&\le&C\left(\|f\|_{\FF^1_T}
	+\|g\|_{\FF^2_T} 
	+\|v_0\|_{\FF^4_T}\right)
	\qquad(0\le\del,\sigma\le R).
	\label{est_v_1}
\end{eqnarray}
By the same argument we also have
\begin{eqnarray}
	\|\rho_{1,E}\|_{_0\EE^1_T}
	&\le&C\left(\|\rho_1\|_{\FF^2_T}
	+\|e^{-|y|(1-\Delta_x)^{1/2}}\rho_0\|_{\rcvitp}\right)
	\nonumber\\
	&\le&C\left(\|\rho_1\|_{\EE^2_T(0,0)}
	+\|\rho_0\|_{\rcsi}\right)
	\qquad(0\le\del,\sigma\le R),
	\label{est_rho_1e}
\end{eqnarray}
where we used Remark~\ref{Bemerkungen}(b) and the 
embeddings 
$W^{2-2/p}_p(\R^n)\hookrightarrow W^{2-3/p}_p(\R^n)$ and
 $\EE^2_T(0,0)\hookrightarrow\FF^2_T$.
Lemma~\ref{ext_ext} implies for $\rho_1$, 
\begin{eqnarray}
	&&\hspace{-5mm}\|\rho_1\|_{_0\EE^2_T(0,0)}\nonumber\\
	&&\hspace{-5mm}\le C\left(\|\rho_0\|_{W^{2-2/p}_p(\R^n)}
	+\|h(0)-[\![c\gamma\partial_y (v_0
	-a e^{-|y|(1-\Delta_x)^{1/2}}\rho_0]\!]\|_{W^{1-3/p}_p(\R^n)}
	\right)\nonumber\\
	&&\hspace{-5mm}\le C\left(\|\rho_0\|_{W^{2-2/p}_p(\R^n)}
	+\|h\|_{\FF^3_T}
	+\|v_0\|_{\FF^4_T}
	+\|(1-\Delta_x)^{1/2}\rho_0]\|_{W^{1-3/p}_p(\R^n)}
	\right)\nonumber\\
	&&\hspace{-5mm}\le C\left(\|\rho_0\|_{W^{2-2/p}_p(\R^n)}
	+\|h\|_{\FF^3_T}
	+\|v_0\|_{\FF^4_T}
	\right)
	\label{est_rho_1_a}
\end{eqnarray}
and 
\begin{eqnarray}
	&&\hspace{-5mm}\sigma\|\rho_1\|_{_0\EE^2_T(0,1)}\nonumber\\
	&&\hspace{-5mm}\le C\left(\sigma\|\rho_0\|_{W^{4-3/p}_p(\R^n)}
	+\sigma\|h(0)-[\![c\gamma\partial_y (v_0
	-a e^{-|y|(1-\Delta_x)^{1/2}}\rho_0)]\!]\|_{W^{2-6/p}_p(\R^n)}
	\right)\nonumber\\
	&&\hspace{-5mm}\le C\left(\sigma\|\rho_0\|_{W^{4-3/p}_p(\R^n)}
	+\sigma\|h(0)-[\![c\gamma\partial_y v_0]\!]\|_{W^{2-6/p}_p(\R^n)}
\right)
\label{est_rho_1_b}
\end{eqnarray}
as well as
\begin{eqnarray}
	&&\del\|\rho_1\|_{_0\EE^2_T(1,0)}\nonumber\\
	&&\le C\left(\del\|\rho_0\|_{W^{4-3/p}_p(\R^n)}
	+\del\|h(0)-[\![c\gamma\partial_y (v_0
	-a e^{-|y|(1-\Delta_x)^{1/2}}\rho_0)]\!]\|_{W^{2-3/p}_p(\R^n)}
	\right)\nonumber\\
	&&\le C\left(\del\|\rho_0\|_{W^{4-3/p}_p(\R^n)}
	+
	\|g\|_{\FF^2_T}+\|v_0\|_{\FF^4_T}
	+\sigma\|\rho_0\|_{\rcsisu}\right),
	\label{est_rho_1_c}
\end{eqnarray}
for $0\le\del,\sigma\le R$,
where we used in \eqref{est_rho_1_c} once again compatibility
condition \eqref{assum_ml_b}.
Inserting \eqref{est_rho_1_a} into \eqref{est_rho_1e}
we obtain by \eqref{est_v_1}--\eqref{est_rho_1_c} that
\begin{eqnarray}
	\|(v_1,\rho_1,\rho_{1,E})\|_{\EE_T(\del,\sigma)}
	&\le& C\left(
	\|(f,g,h,v_0,\rho_0)\|_{\FF_T(0,0)}
	+(\del+\sigma)\|\rho_0\|_{W^{4-3/p}_p(\R^n)}
	\right.\nonumber\\
	&&\left.\strut+
	\sigma\|h(0)-[\![c\gamma\partial_y v_0 ]\!]\|_{W^{2-6/p}_p(\R^n)}
	\right)
	\label{est_all_3}
\end{eqnarray}
for $0\le\del,\sigma\le R$.
Inserting \eqref{est_all_3} into 
\eqref{est_red_sol} we can derive exactly the same estimate
for $(v_2,\rho_2,\rho_{2,E})$.
Combining the estimates for 
$(v_1,\rho_1,\rho_{1,E})$ and $(v_2,\rho_2,\rho_{2,E})$
we finally arrive at \eqref{max_reg_ineq_inhom} and the
proof is complete.
\end{proof}
Next we prove convergence for the solutions of
problem \eqref{LTS}--\eqref{2-rho_E}, that is, 
Theorem~\ref{conv_main_lin}. 
\begin{proof} (of Theorem~\ref{conv_main_lin}.)\\
We employ the decomposition
\[
	\rho^\mu=\rho_1^\mu+\rho_2^\mu
\]
as given in \eqref{splitting}.
We have to show that
\begin{itemize}
\item[(i)] $(v_1^\mu,\rho_1^\mu,\rho_{1,E}^\mu)\to
	(v_1^{\mu_0},\rho_1^{\mu_0},\rho_{1,E}^{\mu_0})$
	\ in \ $\EE_T(\mu_0)$,
\medskip
\item[(ii)] $(v_2^\mu,\rho_2^\mu,\rho_{2,E}^\mu)\to
	(v_2^{\mu_0},\rho_2^{\mu_0},\rho_{2,E}^{\mu_0})$
	\ in \ $\EE_T(\mu_0)$.
\end{itemize}
(i)\ \ We start with proving  convergence of $\rho_1^\mu$. 
This function is according to \eqref{ABC} an extension
of the traces
\[
	\left(\rho_1^\mu(0),\partial_t\rho_1^\mu(0)\right):=
	\left(\rho_0^\mu,q_0^\mu\right),
\]
where we set
\begin{equation}\label{def_q}
	q_0^\mu:=h^\mu(0)-[\![c\gamma\partial_y(v_0^\mu
	-a e^{-|y|(1-\Delta_x)^{1/2}}\rho_0^\mu)]\!].
\end{equation}
Since the extension operator in Lemma~\ref{ext_ext} is linear
and independent of $\mu$ we can estimate for all $\mu\in I_0$, 
\begin{eqnarray}
	\|\rho_1^{\mu}-\rho_1^{\mu_0}\|_{\EE^2_T(\mu_0)}
	&\le& C\left(\|\rho_0^{\mu}-\rho_0^{\mu_0}\|_{\FF^5_T(\mu_0)}
	   +\|q_0^\mu-q_0^{\mu_0}\|_{\FF^6_T(\mu_0)}\right),
	\label{66}
\end{eqnarray}
where
\[
	\FF^6_T(\mu_0):=W^{1-3/p}_p(\R^n)
	    \cap W^{\sg(\sigma_0)(2-6/p)}_p(\R^n)
	\cap W^{\sg(\del_0)(2-3/p)}_p(\R^n).
\]
It is clear by \eqref{conv_ass_inhom_1}
that the first term on the right hand side of
\eqref{66} tends to zero.
In order to see the convergence of the second 
term we distinguish the three cases $\del_0=\sigma_0=0$,
and $\del_0>0,\sigma_0\ge 0$, and $\del_0=0,\sigma_0> 0$.\\[1mm]
The case $\del_0=\sigma_0=0$: \ Here we have 
$\FF^6_T(\mu_0)=W^{1-3/p}_p(\R^n)$ and we obtain by a direct
estimate and \eqref{conv_ass_inhom_1} that
\begin{eqnarray*}
	&&\|q_0^\mu-q_0^{\mu_0}\|_{W^{1-3/p}_p(\R^n)}\\
	&&\le C\left(
	      \|h^\mu-h^{\mu_0}\|_{\FF^3_T}
	       +\|v_0^\mu-v_0^{\mu_0}\|_{\FF^4_T}
	      +\|\rho_0^\mu-\rho_0^{\mu_0}\|_{W^{2-3/p}_p(\R^n)}
	      \right)\\
  	&&\to 0\qquad\qquad(\mu\to\mu_0).
\end{eqnarray*}
The case $\del_0>0,\sigma_0\ge 0$: \ Then 
$\FF^6_T(\mu_0)=W^{2-3/p}_p(\R^n)$. In this case we can employ
compatibility condition \eqref{assum_ml_b} in Theorem~\ref{main_lin}
which results in
\begin{eqnarray}
	&&\hspace{-8mm}\|q_0^\mu-q_0^{\mu_0}\|_{W^{2-3/p}_p(\R^n)}\nonumber\\
	&&\hspace{-8mm} =\left\|\frac{1}{\del}\left(
	    g^\mu(0)-\gamma v_0^\mu+\sigma\Delta_x
	    \rho_0^\mu\right)
	    -\frac{1}{\del_0}\left(
	    g^{\mu_0}(0)-\gamma v_0^{\mu_0}+\sigma\Delta_x
	    \rho_0^{\mu_0}\right)
	    \right\|_{W^{2-3/p}_p(\R^n)}\nonumber\\
	&&\hspace{-8mm}    \le C\left(
	    \left\|\frac{1}{\del}g^\mu
	    -\frac{1}{\del_0}g^{\mu_0}\right\|_{\FF^2_T}
	    +\left\|\frac{1}{\del}v_0^\mu
	    -\frac{1}{\del_0}v_0^{\mu_0}\right\|_{\FF^4_T}
	    +\left\|\frac{\sigma}{\del}\rho_0^\mu
	    -\frac{\sigma_0}{\del_0}
	    \rho_0^{\mu_0}\right\|_{\FF^5_T(\mu_0)}\right).
	\label{456}
\end{eqnarray}
In view of $\del_0>0$ observe that $\rho_0^\mu\to
\rho_0^{\mu_0}$ in $\FF^5_T(\mu_0)=W^{4-3/p}_p(\R^n)$ 
by \eqref{conv_ass_inhom_1}. This yields
\begin{eqnarray*}
	\left\|\frac{\sigma}{\del}\rho_0^\mu
	    -\frac{\sigma_0}{\del_0}
	    \rho_0^{\mu_0}\right\|_{\FF^5_T(\mu_0)}
	&\le& \frac{\sigma}{\del}
	    \|\rho_0^\mu
	    -\rho_0^{\mu_0}\|_{\FF^5_T(\mu_0)}
	    +\left(\frac{\sigma}{\del}-\frac{\sigma_0}{\del_0}\right)
	    \|\rho_0^{\mu_0}\|_{\FF^5_T(\mu_0)}.\\
	&\to& 0\qquad\qquad(\mu\to\mu_0).
\end{eqnarray*}
In the same way we see that the first and the second term 
on the right hand side of \eqref{456} vanish for $\mu\to\mu_0$.
\\[1mm]
The case $\del_0=0,\sigma_0> 0$: \ 
Since $\del\to 0$, here we cannot apply compatibility condition
\eqref{assum_ml_b}. 
This leads to condition \eqref{conv_ass_inhom_2} 
in the statement of the theorem.
In fact, here we obtain
\begin{align*}
	\|q_0^\mu-q_0^{\mu_0}\|_{W^{2-6/p}_p(\R^n)}
	\le&\, C\left(\left\|h^\mu(0)-[\![c\gamma\partial_y v_0^\mu ]\!]
	    -h^{\mu_0}(0)+ [\![c\gamma\partial_y v_0^{\mu_0}]\!]
	    \right\|_{W^{2-6/p}_p(\R^n)}\right.\\
	& \left.\strut 
	      +\|\rho_0^\mu-\rho_0^{\mu_0}\|_{W^{3-3/p}_p(\R^n)}
	      \right).
\end{align*}
It is clear that for $\sigma_0>0$ condition 
\eqref{conv_ass_inhom_2} implies that the first term on the 
right hand side vanishes, whereas the second term 
tends to zero again by \eqref{conv_ass_inhom_1}. 

Also here the convergence of 
$\rho_{1,E}^\mu$ follows by the convergence of $\rho_1^\mu$
in view of the fact that $\rho_{1,E}^\mu$ is the solution of 
\eqref{2-rho_E} with $\rho$ replaced by $\rho_1^\mu$. If
$\T$ denotes again the solution operator of this 
diffusion equation, by \cite[~Proposition 5.1]{EPS03} we 
obtain 
\begin{eqnarray}
	\|\rho_{1,E}^\mu-\rho_{1,E}^{\mu_0}\|_{\EE^1_T}
	&=&\|\T(0,\rho_1^\mu-\rho_1^{\mu_0},
		e^{-|y|(1-\Delta_x)^{-1/2}}
	      (\rho_0^\mu-\rho_0^{\mu_0}))\|_{\EE^1_T}
	      \nonumber\\
	&\le& C\left(
	      \|\rho_1^\mu-\rho_1^{\mu_0}\|_{\FF^2_T}
	      +\|e^{-|y|(1-\Delta_x)^{-1/2}}
	      (\rho_0^\mu-\rho_0^{\mu_0})\|_{\FF^4_T}
	      \right)\nonumber\\
	&\le& C\left(
	      \|\rho_1^\mu-\rho_1^{\mu_0}\|_{\EE^2_T(0,0)}
	      +\|\rho_0^\mu-\rho_0^{\mu_0}\|_{W^{2-3/p}_p(\R^n)}
	      \right)\nonumber\\
	&\to& 0\qquad\qquad (\mu\to\mu_0),
	\label{est_35}
\end{eqnarray}
by the just proved convergence of $\rho_1^\mu$ and 
\eqref{conv_ass_inhom_1}.

Observe that $v_1^\mu$ is, according to
\eqref{diffusion}, the solution of 
the same diffusion equation with right hand side
$(f^\mu,g^\mu+e^{-(1-\Delta_x)t}(\gamma v_0^\mu-g^\mu(0)),v_0^\mu)$ for 
$\mu\in I_0$.
Moreover, we have that
\begin{eqnarray*}
	&&\|e^{-(1-\Delta_x)t}(\gamma v_0^\mu-g^\mu(0)
	-\gamma v_0^{\mu_0}+g^{\mu_0}(0))\|_{\FF^2_T}\\
	&&\le C\|\gamma v_0^\mu-g^\mu(0)
	      -\gamma v_0^{\mu_0}+g^{\mu_0}(0)\|_{W^{2-3/p}_p(\R^n)}\\
	&&\le C\left(\|v_0^\mu-v_0^{\mu_0}\|_{\FF^4_T}
	      +\|g^\mu-g^{\mu_0}\|_{\FF^2_T}\right).
\end{eqnarray*}
Hence we obtain 
\begin{eqnarray*}
	\|v_1^\mu-v_1^{\mu_0}\|_{\EE^1_T}
	&\le& C\left(\|f^\mu-f^{\mu_0}\|_{\FF^1_T}
	      +\|v_0^\mu-v_0^{\mu_0}\|_{\FF^4_T}
	      +\|g^\mu-g^{\mu_0}\|_{\FF^2_T}\right)\nonumber\\
	&\to& 0\qquad\qquad (\mu\to\mu_0)
	\label{conv_v_1}
\end{eqnarray*}
by \eqref{conv_ass_inhom_1}, and (i) is proved.\\[2mm]
(ii) \ \ 
Note that $(v^\mu_2,\rho^\mu_2,\rho_{2,E}^\mu)$ is the solution
of \eqref{reduced}--\eqref{ABB}. According to 
Theorem~\ref{conv_homog_ini} it therefore suffices to prove
convergence for the corresponding data. To be precise,
it remains to show that
\begin{equation}\label{conv_red_g}
	\tg^\mu\to \tg^{\mu_0}
	\quad\mbox{in}\quad {_0\FF^2_T}
	\qquad (\mu\to\mu_0),
\end{equation}
where 
\[
	\tg^\mu=\sigma\Delta_x\rho^\mu_1-\del\partial_t\rho^\mu_1 
	-e^{-(1-\Delta_x)t}\zeta^\mu,
\]
and that
\begin{equation}\label{conv_red_h}
	\th^\mu\to \th^{\mu_0}
	\quad\mbox{in}\quad {_0\FF^3_T}
	\qquad (\mu\to\mu_0),
\end{equation}
where
\[
	\th^\mu=h^\mu-\partial_t\rho^\mu_1 
	-[\![c\gamma\partial_y(v^\mu_1-a(\mu)\rho^\mu_{1,E})]\!].
\]
First we estimate
\begin{eqnarray*}
	\|\th^\mu-\th^{\mu_0}\|_{\FF^3_T}
	&\le& C\left(
	      \|h^\mu-h^{\mu_0}\|_{\FF^3_T}
	      +\|v_1^\mu-v_1^{\mu_0}\|_{\EE^1_T}\right.\\
	&&    \left.\strut    + \|a(\mu)\rho_{1,E}^\mu
	      -a(\mu_0)\rho_{1,E}^{\mu_0}\|_{\EE^1_T}
	      +\|\rho_1^\mu-\rho_1^{\mu_0}\|_{\EE^2_T(0,0)}
	      \right),
\end{eqnarray*}
and we see that \eqref{conv_red_h} follows from
(i), \eqref{cont_cond_a}, and \eqref{conv_ass_inhom_1}.
For $\tg^\mu$ we have 
\begin{eqnarray}
	\|\tg^\mu-\tg^{\mu_0}\|_{\FF^2_T}
	&\le& C\left(
	      \|\del\rho_1^\mu-\del_0\rho_1^{\mu_0}
	      \|_{\EE^2_T(1,0)}
	      +\|\sigma\rho_1^\mu-\sigma_0\rho_1^{\mu_0}
	      \|_{\EE^2_T(0,1)}
	      \right.\nonumber\\
	&&    \left.\strut    
	      +\|\zeta^\mu-\zeta^{\mu_0}\|_{W^{2-3/p}_p(\R^n)}
	      \right).
	      \label{3456}
\end{eqnarray}
By employing the convergence assumptions also here we will prove
that each single term on the right hand side of
\eqref{3456} tends to zero for $\mu\to \mu_0$.
In view of \eqref{def_zeta}
and \eqref{conv_ass_inhom_1} the convergence of the third term
in \eqref{3456} is clear. 
The first two terms are more involved. In fact, this is the
point where assumption \eqref{conv_ass_inhom_3} enters.
In analogy to (i) we again distinguish the three cases
$\del_0=\sigma_0=0$,
and $\del_0>0,\sigma_0\ge 0$, and $\del_0=0,\sigma_0> 0$.\\[1mm]
The case $\del_0=\sigma_0=0$: \ 
Note that in this case condition \eqref{assum_ml_b} for
$\mu_0$ turns into
\begin{equation}\label{1244}
	\gamma v_0^{\mu_0}-g^{\mu_0}(0)=0.
\end{equation}
By using this fact, Lemma~\ref{ext_ext}, \eqref{assum_ml_b}
for $\mu$, and recalling that $q_0^\mu$ still denotes 
the function defined in \eqref{def_q} we obtain 
\begin{eqnarray*}
	&&\|\del\rho_1^\mu\|_{\EE^2_T(1,0)}\\
	&&\le C\left(\del\|\rho_0^\mu\|_{\rcsisu}
	      +\del\|q_0^\mu\|_{W^{2-3/p}_p(\R^n)}
	      \right)\\
	&&\le C\left(\del\|\rho_0^\mu\|_{\rcsisu}
	      +\|g^{\mu}(0)-\gamma v_0^{\mu}
	        -g^{\mu_0}(0)+\gamma v_0^{\mu_0}
		+\sigma\Delta_x\rho_0^\mu\|_{W^{2-3/p}_p(\R^n)}
	       \right)\\
	&&\le C\left((\del+\sigma)\|\rho_0^\mu\|_{\rcsisu}
	      +\|g^{\mu}-g^{\mu_0}\|_{\FF^2_T}
	      +\|v_0^{\mu}-v_0^{\mu_0}\|_{\FF^4_T}
	       \right).
\end{eqnarray*}
In view of \eqref{conv_ass_inhom_1} 
and \eqref{conv_ass_inhom_3} we conclude that
\[
	\|\del\rho_1^\mu\|_{\EE^2_T(1,0)}
	\to 0\qquad\quad(\mu\to\mu_0).
\]
For the second term in \eqref{3456}  Lemma~\ref{ext_ext}
yields
\begin{eqnarray*}
	\|\sigma\rho_1^\mu\|_{\EE^2_T(0,1)}
	&\le& C\left(\sigma\|\rho_0^\mu\|_{\rcsisu}
	      +\sigma\|q_0^\mu\|_{W^{2-6/p}_p(\R^n)}
	      \right)\\
	&\le& C\left(\sigma\|\rho_0^\mu\|_{\rcsisu}
	      +\sigma\|h^\mu(0)-[\![c\gamma\partial_y
	      v_0^\mu ]\!]\|_{W^{2-6/p}_p(\R^n)}
	       \right).
\end{eqnarray*}
Hence, if $\del>0$, it follows 
\begin{equation}\label{565}
	\|\sigma\rho_1^\mu\|_{\EE^2_T(0,1)}
	\to 0\qquad\quad(\mu\to\mu_0)
\end{equation}
by \eqref{conv_ass_inhom_2} and \eqref{conv_ass_inhom_3}.
If $\del=0$, we have $\sigma\Delta_x\rho_0^\mu
=\gamma v_0^\mu-g^\mu(0)$. This yields
\begin{eqnarray*}
	\|\sigma\rho_0^\mu\|_{\rcsisu}
	&\le& C\left(\sigma\|\rho_0^\mu\|_{\rcsi}
	      +\|\sigma\Delta_x\rho^\mu_0\|_{W^{2-3/p}_p(\R^n)}
	      \right)\\
	&\le& C\left(\sigma\|\rho_0^\mu\|_{\rcsi}
	       +\|g^{\mu}-g^{\mu_0}\|_{\FF^2_T}
	      +\|v_0^{\mu}-v_0^{\mu_0}\|_{\FF^4_T}
	       \right),
\end{eqnarray*}
where we used again \eqref{1244}. Observe that
$\|\rho_0^\mu\|_{\rcsi}$ is uniformly bounded in $\mu\in I_0$
by assumption \eqref{conv_ass_inhom_1}. Thus, in this case 
\eqref{565} is obtained as a 
consequence of \eqref{assum_ml_b}, \eqref{conv_ass_inhom_1},
and \eqref{conv_ass_inhom_2}.\\[1mm]
The case $\del_0=0,\sigma_0>0$: \ \ 
Here we have
\begin{equation}\label{1233}
	\gamma v_0^{\mu_0}-\sigma\Delta_x\rho_0^{\mu_0}
	=g^{\mu_0}(0).
\end{equation}
In a similar way as in the previous case we deduce, if $\del>0$,
that
\begin{equation}
\label{332}
\begin{split}
	\|\del\rho_1^\mu\|_{\EE^2_T(1,0)}
	\le&\, C\left(\del\|\rho_0^\mu\|_{\rcsisu}
	      +\|\sigma\rho_0^{\mu}-\sigma_0\rho_0^{\mu_0}
	      \|_{\rcsisu}\right.\\
	&    \left.\strut +\|g^{\mu}-g^{\mu_0}\|_{\FF^2_T}
	      +\|v_0^{\mu}-v_0^{\mu_0}\|_{\FF^4_T}
	       \right).
\end{split}
\end{equation}
Note that in the case $\sigma_0>0$ we also have that
\[
	\rho_0^\mu\to\rho_0^{\mu_0}
	\quad\mbox{in}\quad\FF^5_T(\mu_0)=\rcsisu
	\qquad\quad(\mu\to\mu_0).
\]
By this fact it is easy to see that the first two terms
in \eqref{332} vanish for $(\mu\to\mu_0)$,
whereas the convergence of the last two terms
follows again by \eqref{conv_ass_inhom_1}.
That the second term in \eqref{3456} tends to zero
here follows easily from the inequality
\[
	\|\sigma\rho_1^{\mu}-\sigma_0\rho_1^{\mu_0}
	      \|_{\EE^2_T(0,1)}
	\le \frac{\sigma}{\sigma_0}
	    \|\rho_1^{\mu}-\rho_1^{\mu_0}
	      \|_{\EE^2_T(0,\sigma_0)}
	    +\frac{\sigma-\sigma_0}{\sigma_0}
	     \|\rho_1^{\mu_0}
	      \|_{\EE^2_T(0,\sigma_0)}
\]
and the convergence of $\rho_1^\mu$ 
in $\EE^2_T(\mu_0)=\EE^2_T(0,\sigma_0)$ proved in (i).
Observe that the last argument also implies convergence for the case
$\del=0$, since then the first term in \eqref{3456} vanishes completely.
\\[1mm]
The case $\del_0>0,\sigma_0\ge 0$: \ \
Here the convergence of the first term in \eqref{3456}
follows completely analogous to the convergence of the
second term in the previous case.
If we suppose that also $\sigma_0>0$ the convergence
of the second term in \eqref{3456} follows by the 
same argument. In the case that $\sigma_0=0$ also here
an application of Lemma~\ref{ext_ext}
implies	
\begin{eqnarray*}
	\|\sigma\rho_1^\mu\|_{\EE^2_T(0,1)}
	&\le& C\left(\sigma\|\rho_0^\mu\|_{\rcsisu}
	      +\sigma\|q_0^\mu\|_{W^{2-6/p}_p(\R^n)}
	      \right)
\end{eqnarray*}
Moreover, we still have $\FF^5_T(\mu_0)=\rcsisu$, which implies
$\sigma\|\rho_0^\mu\|_{\rcsisu}\to 0$ for $\mu\to\mu_0$
by \eqref{conv_ass_inhom_1}. For the second term on the right
hand side of the above inequality note that
from the case $\del_0>0,\sigma\ge 0$ in (i) we know that
$q_0^\mu\to q_0^{\mu_0}$ in $W^{2-3/p}_p(\R^n)$. 
This implies that this term vanishes as well for $\mu\to\mu_0$.
Hence also in this case we have that
\[
	\sigma\|\rho_1^\mu\|_{\EE^2_T(0,1)}
	\to 0\qquad\quad (\mu\to\mu_0).
\]

The three cases together show that 
\[
	(0,\tg^\mu,\th^\mu)
	\to (0,\tg^{\mu_0},\th^{\mu_0})
	\quad\mbox{in}\quad{_0\FF_T}
	\qquad\quad (\mu\to\mu_0),
\]
and therefore Theorem~\ref{conv_homog_ini}
implies (ii).
\end{proof}


\section{Appendix}

The reduction of problem \eqref{LTS}-\eqref{2-rho_E} 
to the case of vanishing traces
in the proof of Theorem~\ref{main_lin} 
was based on the following two results.
Observe that the assertions in Lemma~\ref{lem_ext_1} follow
directly from the general trace result \cite[Theorem~4.5]{dss2008}.
However, for the sake of completeness and for a better understanding
of the proof of subsequent Lemma~\ref{ext_ext} we give its proof
here.
In the following we adopt the notation of Section~\ref{inhom_traces}.
\begin{lemma}
\label{lem_ext_1}
Let $1<p<\infty$, $T\in(0,\infty]$, and $J=(0,T)$.
\begin{itemize}
\item[(i)] For each $\eta_0\in\rcsisu$ there exists an extension
           \begin{equation*}
                \qquad\qquad\eta\in\EE^2_T(1,1)
           \end{equation*}
           such that $\eta_1(0)=\sigma_0$ and, if $p>3$, also that
           $\partial_t\eta_1(0)=0$.
\item[(ii)] Suppose $p>3/2$. Then for each $\eta_0\in\rcsisu$ and 
           $\eta_1\in W^{2-3/p}_p(\R^n)$ 
           there exists an extension $\eta\in \EE^2_T(1,1)$ 
           satisfying $\eta(0)=\eta_0$, 
	   $\partial_t\eta(0)=\eta_1$ and the estimate
\[
	\|\eta\|_{\EE^2_T(1,1)}
	\le C\left(\|\eta_0\|_{\rcsisu}
	+\|\eta_1\|_{W^{2-3/p}_p(\R^n)}\right).
\]
\item[(iii)] Suppose $p>3$. Then for each $\eta_0\in\rcsisu$ and 
           $\eta_1\in W^{2-6/p}_p(\R^n)$ 
           there exists an extension $\eta\in \EE^2_T(0,1)$ 
           satisfying $\eta(0)=\eta_0$, 
	   $\partial_t\eta(0)=\eta_1$ and the estimate
\[
	\|\eta\|_{\EE^2_T(0,1)}
	\le C\left(\|\eta_0\|_{\rcsisu}
	+\|\eta_1\|_{W^{2-6/p}_p(\R^n)}\right).
\]
\end{itemize}
\end{lemma}
\begin{proof}
(i)
Let $1<p<\infty$. We claim that
\begin{equation}
\label{eta-1}
        \eta(t)
        :=(2\e^{-t(1-\Delta_x)}
          -\e^{-2t(1-\Delta_x)})\eta_0
\end{equation}  
satisfies the properties asserted in (i).
We have
\[
        \e^{-kt(1-\Delta_x)}\eta_0
        \in \W^1_p(J,W^{2-1/p}_p(\R^n))
        \cap L_p(J,W^{4-1/p}_p(\R^n))
\]
for $k=1,2$. It is a consequence of the mixed derivative theorem 
that the latter space is continuously embedded in 
$W^{1-1/2p}_p(J,\W^2_p(\R^n))$. This implies 
that
\[
        \partial_t\e^{-kt(1-\Delta_x)}\eta_0
        =   -k(1-\Delta_x)
            \e^{-kt(1-\Delta_x)}\eta_0
        \in W^{1-1/2p}_p(J,L_p(\R^n)).
\]
Consequently,
\[
        \eta\in W^{2-1/2p}_p(J,L_p(\R^n))
\]
and we have that
\[
	\|\eta\|_{W^{2-1/2p}_p(J,L_p(\R^n))}
	\le C\|\eta\|_{\W^1_p(J,W^{2-1/p}_p(\R^n))
        \cap L_p(J,W^{4-1/p}_p(\R^n))}.
\]
The maximal regularity of $(1-\Delta_x)$ on $W^{2-1/p}_p(\R^n)$
and the embedding
\[
	\EE^2_T(1,1)\hookrightarrow \W^1_p(J,W^{2-1/p}_p(\R^n))
\]
then yields 
\[
	\|\eta\|_{\EE^2_T(1,1)}
	\le C\|\eta_0\|_{\rcsisu}.
\]
Obviously $\eta(0)=\eta_0$. If $p>3$, the time trace
of $\partial_t\eta$ is well defined and we
also have $\partial_t\eta(0)=0$. This proves (i).
\smallskip\\
(ii)
Now suppose $p>3/2$. Here we first set
\begin{equation}
\label{eta-2}
        \teta(t):=
          (\e^{-t(1-\Delta_x)}-\e^{-2t(1-\Delta_x)})(1-\Delta_x)^{-1}
	  \eta_1,
\end{equation}
Then for $\eta_1\in W^{2-3/p}_p(\R^n)$ we have that
\[
        \e^{-kt(1-\Delta_x)}(1-\Delta_x)^{-1}\eta_1
        \in \W^1_p(J,W^{2-1/p}_p(\R^n))
            \cap L_p(J,W^{4-1/p}_p(\R^n))
\]
for $k=1,2$. 
By virtue of the embedding
\[
	\W^1_p(J,W^{2-1/p}_p(\R^n))
            \cap L_p(J,W^{4-1/p}_p(\R^n))
        \hookrightarrow \W^{1-1/2p}_p(J,W^2_p(\R^n))
\]
we obtain 	
\[
        \partial_t\e^{-kt(1-\Delta_x)}(1-\Delta_x)^{-1}\eta_1
        \in \W^{1-1/2p}_p(J,L_p(\R^n)),
\]
hence that
\[
        \e^{-kt(1-\Delta_x)}(1-\Delta_x)^{-1}\eta_1
        \in \W^{2-1/2p}_p(J,L_p(\R^n)).
\]
By the same arguments as in (i) we obtain the 
estimate 
\[
	\|\teta\|_{\EE^2_T(1,1)}
	\le C\|\eta_1\|_{W^{2-3/p}_p(\R^n)}.
\]
If $\bareta$ denotes the extension constructed in (i),
then 
\[
	\eta:=\bareta+\teta
\]
satisfies the regularity assertions in (ii).
That $\eta(0)=\eta_0$ and $\partial_t\eta(0)=\eta_1$
is obvious.
\smallskip\\
(iii)
Now we set
\begin{equation}
\label{eta-22}
        \teta(t):=
          (\e^{-t(1-\Delta_x)^2}-\e^{-2t(1-\Delta_x)^2})
	  (1-\Delta_x)^{-2}\eta_1.
\end{equation}
We have to check that $\e^{-kt(1-\Delta_x)^2}
(1-\Delta_x)^{-2}\eta_1\in\EE^2_T(0,1)$.
In view of $\eta_1\in W^{2-6/p}_p(\R^n)$ we have that
\[
        \e^{-kt(1-\Delta_x)^2}(1-\Delta_x)^{-2}\eta_1
        \in \W^1_p(J,W^{2-2/p}_p(\R^n))
            \cap L_p(J,W^{6-2/p}_p(\R^n)).
\]
From the embedding
\[
	\W^1_p(J,W^{2-2/p}_p(\R^n))
            \cap L_p(J,W^{6-2/p}_p(\R^n))
        \hookrightarrow \W^{1/2-1/2p}_p(J,W^4_p(\R^n))
\]
we infer 
\[
        \partial_t\e^{-kt(1-\Delta_x)^2}(1-\Delta_x)^{-2}\eta_1
        \in \W^{1/2-1/2p}_p(J,L_p(\R^n)),
\]
and therefore that
\[
        \e^{-kt(1-\Delta_x)}(1-\Delta_x)^{-2}\eta_1
        \in \W^{3/2-1/2p}_p(J,L_p(\R^n)).
\]
Then $\eta:=\bareta+\teta$ satisfies all
the assertions claimed in (iii), where $\bareta$ denotes again
the extension obtained in (i).
\end{proof}
Lemma~\ref{lem_ext_1} in conbination with \cite[Lemma~6.4]{PSS07} 
yields the following result which provides a simultaneous
extension for different regularity assumptions on the traces.
\begin{lemma}
\label{ext_ext}
Let $3<p<\infty$, $T\in(0,\infty]$, and $J=(0,T)$.
For $\eta_0$ and $\eta_1$ there exists an (simultaneous) extension
function $\eta$ such that $\eta(0)=\eta_0$, $\partial_t\eta(0)=\eta_1$,
and
\[
	\|\eta\|_{\EE^2_T(0,0)}
	\le C\left(\|\eta_0\|_{\rcsi}
	+\|\eta_1\|_{W^{1-3/p}_p(\R^n)}\right),
\]
if $(\eta_0,\eta_1)\in \rcsi\times W^{1-3/p}_p(\R^n)$,
\[
	\|\eta\|_{\EE^2_T(0,1)}
	\le C\left(\|\eta_0\|_{\rcsisu}
	+\|\eta_1\|_{W^{2-6/p}_p(\R^n)}\right),
\]
if $(\eta_0,\eta_1)\in \rcsisu\times W^{2-6/p}_p(\R^n)$,
and
\[
	\|\eta\|_{\EE^2_T(1,1)}
	\le C\left(\|\eta_0\|_{\rcsisu}
	+\|\eta_1\|_{W^{2-3/p}_p(\R^n)}\right),
\]
if $(\eta_0,\eta_1)\in \rcsisu\times W^{2-3/p}_p(\R^n)$,
with $C>0$ independent of $\eta_0$ and $\eta_1$.
\end{lemma}
\begin{proof}
The idea for obtaining a simultaneous extension function as 
stated in the lemma is to employ a combination 
of the extension operators we used
in Lemma~\ref{lem_ext_1}. More precisely, we claim that
\begin{eqnarray*}
        \eta(t)
        &:=&\left(2\e^{-t(1-\Delta_x)^{1/2}}
          -\e^{-2t(1-\Delta_x)^{1/2}}\right)
	  \left(2\e^{-t(1-\Delta_x)}
          -\e^{-2t(1-\Delta_x)}\right)\eta_0\\
	  &&
	  \strut +\e^{-t(1-\Delta_x)}
	  \left(\e^{-t(1-\Delta_x)^2}
          -\e^{-2t(1-\Delta_x)^2}\right)(1-\Delta_x)^{-2}\eta_1
\end{eqnarray*}  
satisfies all the properties asserted.
Observe that $(e^{-\beta t(1-\Delta_x)^\alpha})_{t\ge 0}$
is a bounded $C_0$-semigroup and $(1-\Delta_x)^{-\alpha}$
is a bounded operator on $W^{r}_p(\R^n)$ for all 
$r,\alpha,\beta\ge 0$ and $1<p<\infty$. Hence,
\[
	e^{-\beta t(1-\Delta_x)^\alpha},
	(1-\Delta_x)^{-\alpha}\in \LL(W^s_p(J,W^r_p(\R^n)))
\]
for all $s,r,\alpha,\beta\ge 0$ and $1<p<\infty$.
If $(\eta_0,\eta_1)\in \rcsi\times W^{1-3/p}_p(\R^n)$,
we therfore may estimate
\begin{eqnarray*}
        \|\eta\|_{\EE^2_T(0,0)}
        &\le& C\left(\|(2\e^{-t(1-\Delta_x)^{1/2}}
          -\e^{-2t(1-\Delta_x)^{1/2}})\eta_0\|_{\EE^2_T(0,0)}\right.\\
	&&\left.\strut 
	+\|\e^{-t(1-\Delta_x)}\eta_1\|_{\EE^2_T(0,0)}\right).
\end{eqnarray*}  
By \cite[Lemma~6.4]{PSS07} the remaining 
extension operators are known to lift the traces into
the class $\EE^2_T(0,0)$, which implies
\[
	\|\eta\|_{\EE^2_T(0,0)}
	\le C\left(\|\eta_0\|_{\rcsi}
	+\|\eta_1\|_{W^{1-3/p}_p(\R^n)}\right).
\]
Hence the first estimate is proved.
If $(\eta_0,\eta_1)\in \rcsisu\times W^{2-6/p}_p(\R^n)$,
we interchange the roles of the semigroups in the
definition of $\eta$. In fact here we obtain as in
Lemma~\ref{lem_ext_1}~(iii), 
\begin{eqnarray*}
        \|\eta\|_{\EE^2_T(0,1)}
        &\le&C\left(\|2\e^{-t(1-\Delta_x)}
          -\e^{-2t(1-\Delta_x)})\eta_0\|\EE^2_T(0,1)\right.\\
	  &&
	  \left.\strut +\|(\e^{-t(1-\Delta_x)^2}
	  -\e^{-2t(1-\Delta_x)^2})
	  (1-\Delta_x)^{-2}\eta_1\|\EE^2_T(0,1)\right)\\
	  &\le&C\left(\|\eta_0\|_{\rcsisu}
	+\|\eta_1\|_{W^{2-6/p}_p(\R^n)}\right).
\end{eqnarray*}  
Analogously we proceed in the third case. Here we treat
the terms of type $\e^{-\beta t(1-\Delta_x)^{1/2}}$ in
front of $\eta_0$ and the terms 
$\e^{-\beta t(1-\Delta_x)^2}$ and $(1-\Delta_x)^{-1}$ 
in front of $\eta_1$ as bounded operators and gain
the desired regularity by the remaining operators as
in Lemma~\ref{lem_ext_1}~(ii).
A straight forward calculation also shows that
$\eta(0)=\eta_0$ and $\partial_t\eta(0)=\eta_1$.
\end{proof}

\medskip
Received xxxx 20xx; revised xxxx 20xx.
\medskip

\end{document}